\documentclass{amsart}

\usepackage{url,amssymb,enumerate,colonequals}

\usepackage{geometry}
\usepackage{mathrsfs} 
\usepackage[section]{placeins}
\usepackage{MnSymbol}
\usepackage{extarrows}
\usepackage{lscape}
\usepackage{circledsteps}
\usepackage[all,cmtip]{xy}
\usepackage{caption}
\usepackage[OT2,T1]{fontenc}
\usepackage{color}
\usepackage{tikzit}

\tikzstyle{white with black border}=[fill=white, draw=black, shape=circle]
\tikzstyle{black with white font}=[white, fill=black, draw=black, shape=circle, font={\small}, inner sep=0pt]
\tikzstyle{black simple style}=[fill=black, draw=black, shape=circle]
\tikzstyle{A7 node}=[fill={rgb,255: red,220; green,47; blue,2}, draw=black, shape=circle]
\tikzstyle{D5 node}=[fill={rgb,255: red,212; green,73; blue,6}, draw=black, shape=circle]
\tikzstyle{D4 node}=[fill={rgb,255: red,228; green,107; blue,6}, draw=black, shape=circle]
\tikzstyle{A3 node}=[fill={rgb,255: red,250; green,163; blue,7}, draw=black, shape=circle]
\tikzstyle{A2 node}=[fill={rgb,255: red,255; green,186; blue,8}, draw=black, shape=circle]
\tikzstyle{A1 node}=[fill={rgb,255: red,244; green,228; blue,9}, draw=black, shape=circle]
\tikzstyle{A1 node}=[fill={rgb,255: red,244; green,228; blue,9}, draw=black, shape=circle]
\tikzstyle{fatpoint}=[circle, draw=black, fill=black, inner sep=0pt, minimum size=4mm]
\tikzstyle{pointlabel}=[white, font=\small, anchor=center]

\tikzstyle{arrow}=[->]
\tikzstyle{new edge style 0}=[-, fill=none]
\tikzstyle{A7 edge}=[-, fill={rgb,255: red,220; green,47; blue,2}, draw={rgb,255: red,220; green,47; blue,2}]
\tikzstyle{D5 edge}=[-, fill={rgb,255: red,212; green,73; blue,6}, draw={rgb,255: red,212; green,73; blue,6}]
\tikzstyle{D4 edge}=[-, fill={rgb,255: red,228; green,107; blue,6}, draw={rgb,255: red,228; green,107; blue,6}]
\tikzstyle{A3 edge}=[-, fill={rgb,255: red,250; green,163; blue,7}, draw={rgb,255: red,250; green,163; blue,7}]
\tikzstyle{A2 edge}=[-, fill={rgb,255: red,255; green,186; blue,8}, draw={rgb,255: red,255; green,186; blue,8}]
\tikzstyle{dashed arrow}=[->, dashed]
\tikzstyle{A1 edge}=[-, fill={rgb,255: red,244; green,228; blue,9}, draw={rgb,255: red,244; green,228; blue,9}]
\tikzstyle{A3 edge white}=[-, fill={rgb,255: red,250; green,163; blue,7}, draw=white]

\usepackage{tikz-cd}
\usepackage{resizegather}
\usepackage[
        colorlinks=true, citecolor=darkestred, linkcolor=darkred, urlcolor=darkred,
        backref=page,
        pdfauthor={Alvaro Gonzalez Hernandez},
]{hyperref}

\usepackage{comment}
\usepackage{multirow}
\usepackage{mathdots}
\usepackage{amssymb}
\usepackage{amsfonts}
\usepackage{amsmath}
\usepackage[nobysame,alphabetic]{amsrefs}

\usepackage{holtpolt}
\usepackage{xcolor}
\usepackage{adjustbox}
\usepackage{colortbl}
\usepackage{multicol}
\usepackage{appendix}
\usepackage{graphicx}
\usepackage{wrapfig}
\usepackage{stackrel}
\usepackage{changepage} 
\usepackage{listings}
\usepackage{xcolor}
\usepackage{float}
\usepackage{tikz-cd} 
\usepackage{pgfplots}
\usepackage{orcidlink}

\newcommand{\bbH}{\mathbb{H}}

\newcommand{\Q}{\mathbb{Q}}

\newcommand{\Z}{\mathbb{Z}}

\newcommand{\calO}{\mathcal{O}}

\newcommand{\FF}{\mathbb{F}}

\DeclareMathOperator{\coker}{coker}

\DeclareMathOperator{\End}{End}

\DeclareMathOperator{\Kum}{Kum}

\DeclareMathOperator{\rank}{rank}

\DeclareMathOperator{\rig}{rig}

\DeclareMathOperator{\Spec}{Spec}

\newcommand{\GL}{\operatorname{GL}}

\newcommand{\SL}{\operatorname{SL}}
\newcommand{\Stab}{\operatorname{Stab}}
\newcommand{\ESL}{\operatorname{ESL}}
\newcommand{\Ktwo}{\Kum(E\times E)}
\newcommand{\Kthree}{(E\times E)/C_3}
\newcommand{\Kfour}{(E\times E)/C_4}
\newcommand{\Ksix}{(E\times E)/C_6}

\numberwithin{equation}{section}

\theoremstyle{plain}
\newtheorem{theorem}{Theorem}[section]
\newtheorem{lemma}[theorem]{Lemma}

\newtheorem{corollary}[theorem]{Corollary}
\newtheorem{proposition}[theorem]{Proposition}

\theoremstyle{definition}
\newtheorem{definition}[theorem]{Definition}
\newtheorem{question}[theorem]{Question}

\theoremstyle{remark}
\newtheorem{remark}[theorem]{Remark}

\newtheorem*{acknowledgements}{Acknowledgements}

\definecolor{darkestblue}{HTML}{03045E}
\definecolor{darkblue}{HTML}{0077B6}
\definecolor{darkestred}{HTML}{c1121f}
\definecolor{darkred}{HTML}{c1121f}
\definecolor{lighterblue}{HTML}{00B4D8}
\definecolor{warwickpurple}{HTML}{440379}
\definecolor{warwickpurple2}{RGB}{128, 35, 173}

\DeclareRobustCommand{\SkipTocEntry}[5]{}

\makeatletter
\newcommand{\xdashrightarrow}[2][]{\ext@arrow 0359\rightarrowfill@@{#1}{#2}}
\newcommand{\xdashleftarrow}[2][]{\ext@arrow 3095\leftarrowfill@@{#1}{#2}}
\newcommand{\xdashleftrightarrow}[2][]{\ext@arrow 3359\leftrightarrowfill@@{#1}{#2}}
\def\rightarrowfill@@{\arrowfill@@\relax\relbar\rightarrow}
\def\leftarrowfill@@{\arrowfill@@\leftarrow\relbar\relax}
\def\leftrightarrowfill@@{\arrowfill@@\leftarrow\relbar\rightarrow}
\def\arrowfill@@#1#2#3#4{%
  $\m@th\thickmuskip0mu\medmuskip\thickmuskip\thinmuskip\thickmuskip
   \relax#4#1
   \xleaders\hbox{$#4#2$}\hfill
   #3$%
}
\makeatother
\pgfplotsset{compat=1.18}
\begin{document}

\begin{abstract}
Building on earlier work of Katsura and Rybakov, we complete the classification, in every characteristic, of all groups 
$G$ acting on an abelian surface $A$ by automorphisms preserving the group law such that the resolution of the quotient $A/G$ is a K3 surface. In order to do so, we study actions of groups with $p\mid|G|$ in characteristics $p=2,3$ and $5$.
Using the theory of rational double points in positive characteristic, we show that the possible ADE singularity types of
$A/G$ are constrained by the requirement that their local fundamental group contains
$G$ as a subgroup, and we determine the singular locus of $A/G$ via the action of $G$
on the $\ell$-adic Tate module of $A$. As a key step in the classification, we
prove that if $A$ is a supersingular abelian surface and $p\mid|G|$, then $A/G$ can
never be a generalised Kummer surface.
Finally, we construct explicit examples of such generalised Kummer surfaces as quotients of products of two elliptic curves.
\end{abstract}

\title{The classification of generalised Kummer surfaces in positive characteristic}

\author{Alvaro Gonzalez-Hernandez\, \orcidlink{0009-0001-4537-807X}}

\address{Mathematics Institute\\
    University of Warwick\\
    CV4 7AL \\
    United Kingdom\\}

\email{
\href{mailto:alvaro.gonzalez-hernandez@warwick.ac.uk}{alvaro.gonzalez-hernandez@warwick.ac.uk} 
}

\thanks{The author is financially supported by the Warwick Mathematics Institute. Website: \url{https://alvarogohe.github.io}}
\maketitle

\section{Introduction}

One of the most important examples of K3 surfaces are Kummer surfaces, which were first described in the nineteenth century by mathematicians such as Hamilton, Cayley and Kummer, while studying the remarkable geometry of quartic surfaces in $\mathbb{P}^3$ with sixteen singularities \cite{Dolgachev2020KummerStudy}. Thanks to the work of Göpel and Borchardt, we now know that these surfaces arise as the quotients of abelian surfaces by the group $C_2$ acting on the surface as the involution that sends every point to its inverse with respect to the group law \cite{Gopel1847TheoriaeLevis}. Kummer surfaces have since become a class of K3 surfaces that we understand well due to the fact that they are finite covers of abelian surfaces, and much of their geometry can be analysed from an arithmetic perspective \cite{Cassels1996Prolegomena2}.\\

The search of other families of K3 surfaces with similarly nice properties naturally leads to the following question:

\begin{question} \label{question}
Are there any other finite groups acting on abelian surfaces such that the desingularisations of the corresponding quotients are K3 surfaces?
\end{question}

Motivated by this question, we have the following definition:

\begin{definition}
Let $A$ be an abelian surface and $G$ a finite group acting on $A$. We say that the quotient $A/G$ is a \textbf{generalised Kummer surface} if the minimal resolution of singularities of $A/G$ is a K3 surface. 
\end{definition}

In general, the desingularisation of the quotient of an abelian surface by the action of a finite group is not a K3 surface. For example, the quotient of an abelian surface by a translation by a torsion point is another abelian surface. Furthermore, the problem also depends on the characteristic of the field over which the surface is defined. For example, in characteristic two, the desingularisation of the quotient of a supersingular abelian surface by the previously described action of $C_2$ is not a K3 surface.\\

In 1987, Katsura characterised which conditions have to be met for $A/G$ to be a generalised Kummer surface (Theorem \ref{Katsura_theorem}, \cite{Katsura1987GeneralizedP}). In particular, one of them is that the action of $G$ must have at least one fixed point. If we translate this point to be the origin, a natural assumption is that the elements of $G$ act as endomorphisms of the abelian surface, which is usually referred to as saying that $\boldsymbol{G}$ \textbf{preserves the group law}.\\

Under this assumption, in the same paper, Katsura classified all possible groups $G$ acting on an abelian surface $A$ such that $A/G$ is a generalised Kummer surface in characteristic zero, and described some additional cases in positive characteristic. In a more recent paper, Rybakov translated these conditions in terms of the action of $G$ on the rational Tate module of $A$ to extend the classification to all fields of characteristic $p$ and all groups $G$ whose order is not divisible by $p$ \cite{Rybakov2024GeneralizedFields}.\\

This paper aims to complete the classification by studying the remaining cases, which are the quotients in characteristic $p$ where $p=2,3$ and $5$, and $p$ divides the order of $G$. The outcome is the following theorem:

\begin{theorem} \label{main_theorem}
Let $A$ be an abelian surface defined over an algebraically closed field $k$ of characteristic $p \geq 0$, and let $G$ be a finite group acting on $A$ and preserving the group law. Then, the quotient $A/G$ is a generalised Kummer surface if and only if $G$, $p$ and $A$ satisfy the conditions listed in the first two columns of the table below, and the action of $G$ on $A$ is as in Theorem \ref{Katsura_theorem}.

\begin{adjustbox}{width=\textwidth}
\begin{minipage}{\textwidth}
\begin{table}[H] 
\centering
\begin{tabular}{|c|cc|c|}
\hline 
\rowcolor[HTML]{FFFFFF}
$G$                    & \multicolumn{2}{c|}{Conditions}                                              & Singularities of $A/G$         \\ \hline
\multirow{3}{*}{$C_2$} & \multicolumn{2}{c|}{$p\neq2$}                                                & $16A_1$                \\ \cline{2-4} 
                       & \multicolumn{1}{c|}{\multirow{2}{*}{$p=2$}}    & $f(A)=2$         & $4 D_4^1$              \\ \cline{3-4} 
                       & \multicolumn{1}{c|}{}                          & $f(A)=1$         & $2 D_8^2$              \\ \hline
\multirow{2}{*}{$C_3$} & \multicolumn{2}{c|}{$p\neq3$}                                                & $9A_2$                 \\ \cline{2-4} 
                       & \multicolumn{1}{c|}{$p=3$}                     & $f(A)=2$     & $3E_6^1$               \\ \hline
\multirow{2}{*}{$C_4$} & \multicolumn{2}{c|}{$p\neq2$}                                                & $4A_3+6A_1$            \\ \cline{2-4} 
                       & \multicolumn{1}{c|}{$p=2$}                     & $f(A)=2$       & $2E_7^3+D_4^1$          \\ \hline
$C_5$                  & \multicolumn{1}{c|}{$p\equiv \pm 2\mod 5$}     & $f(A)=0$ & $5A_4$                 \\ \hline

\multirow{3}{*}{$C_6$} & \multicolumn{2}{c|}{$p\neq2,3$}                                              & $A_5+4A_2+5A_1$        \\ \cline{2-4} 
                       & \multicolumn{1}{c|}{$p=2$}                     & $f(A)=2$       & $E_6^1+D_4^1+4 A_2$    \\ \cline{2-4} 
                       & \multicolumn{1}{c|}{$p=3$}                     & $f(A)=2$       & $E_7^1+E_6^1+5A_1$     \\ \hline
$C_8$                  & \multicolumn{1}{c|}{$p\equiv \pm 3\mod 8$}     & $f(A)=0$ & $2A_7+A_3+3A_1$        \\ \hline
$C_{10}$               & \multicolumn{1}{c|}{$p\equiv \pm 2\mod 5$}     & $f(A)=0$ & $A_9+2A_4+3A_1$        \\ \hline
$C_{12}$               & \multicolumn{1}{c|}{$p\equiv \pm 5\mod 12$}    & $f(A)=0$ & $A_{11}+A_3+2A_2+2A_1$ \\ \hline
\multirow{2}{*}{$Q_8$} & \multicolumn{1}{c|}{\multirow{2}{*}{$p\neq2$}} & $\#Q_8\text{-fixed points}=2$                          & $2D_4+3A_3+2A_1$       \\ \cline{3-4} 
                       & \multicolumn{1}{c|}{}                          & $\#Q_8\text{-fixed points}=4$                              & $4D_4+3A_1$            \\ \hline
\multirow{2}{*}{$Q_{12}$}               & \multicolumn{2}{c|}{$p\neq2,3$}                                              & $D_5+3A_3+2A_2+A_1$    \\ \cline{2-4}
 & \multicolumn{1}{c|}{$p=2$}                     & $f(A)=2$       & $E_8^4+E_7^3+2A_2$          \\ \hline
$Q_{16}$               & \multicolumn{1}{c|}{$p\equiv \pm 3\mod 8$}     & $f(A)=0$ & $2D_6+D_4+A_3+A_1$     \\ \hline
$Q_{20}$               & \multicolumn{1}{c|}{$p\equiv \pm 2\mod 5$}     & $f(A)=0$ & $D_7+A_4+3A_3$         \\ \hline
$Q_{24}$               & \multicolumn{1}{c|}{$p\equiv \pm 5\mod12$}     & $f(A)=0$ & $D_8+D_4+2A_3+A_2$     \\ \hline
\multirow{2}{*}{$\SL_2(\FF_3)$}         & \multicolumn{2}{c|}{$p\neq2,3$}                                              & $E_6+D_4+4A_2+A_1$     \\ 
\cline{2-4} 
                       & \multicolumn{1}{c|}{$p=3$}                     & $f(A)=2$       & $E_8^2+E_6^1+D_4+A_1$               \\ \hline
$\ESL_2(\FF_3)$        & \multicolumn{1}{c|}{$p\equiv \pm 3\mod 8$}     & $f(A)=0$ & $E_7+D_6+A_3+2A_2$     \\ \hline
$\SL_2(\FF_5)$         & \multicolumn{1}{c|}{$p\equiv \pm 2\mod 5$}     & $f(A)=0$ & $E_8+D_4+A_4+2A_2$     \\ \hline
\end{tabular}
\caption{The classification of generalised Kummer surfaces in positive characteristic.}
\label{main_table}
\end{table}
\end{minipage}
\end{adjustbox}
\end{theorem}
\vspace{5pt}
In Section \ref{sec}, we will explain the contributions of Katsura and Rybakov to the classification of generalised Kummer surfaces. One of the key properties that the quotients $A/G$ have to satisfy is that the singularities must be rational double points. We will recall the theory of rational points in positive characteristic and present the key idea of this paper, which is that the possible ADE singularity types of the quotient $A/G$ must satisfy that their local fundamental group contains $G$ as a subgroup.\\

In Section \ref{computingsings_section}, we will see how the singular points of the quotients $A/G$ are determined by the action of $G$ on the $\ell$-adic rational Tate module of $A$, where $\ell\mid |G|$. When $\ell$ equals the characteristic $p$, we refine the existing theory to show that the number of singular points of $A/G$ is determined by the $p$-rank of $A$.\\

These methods do not apply when $A$ is supersingular. Nevertheless, in Section \ref{ss_section} we show that this causes no issue by proving that if $A$ is supersingular and $p\mid |G|$ then $A/G$ can never be a generalised Kummer surface.\newpage

Finally, in Section \ref{sec_examples}, we construct examples for all the new entries in Table \ref{main_table} as quotients of products of two elliptic curves.

\begin{acknowledgements}
 I would like to thank Damiano Testa for his invaluable guidance and support. I would also like to thank Chris Lazda, and Tianchen Zhao for their help, particularly regarding the arguments in Section \ref{ss_section}. This paper also benefited from useful comments from Samir Siksek and Alexei Skorobogatov. 
\end{acknowledgements}

\vspace{5pt}
\section{Characterising generalised Kummer surfaces} \label{sec}
\subsection{Katsura's theorem}
In order to prove Theorem \ref{main_theorem}, we need to determine when the quotient of an abelian surface by a group action is a generalised Kummer surface. As discussed in the introduction, the work of Katsura \cite{Katsura1987GeneralizedP} and Rybakov \cite{Rybakov2024GeneralizedFields} settles most cases of the classification; accordingly, we concentrate on the remaining cases in characteristics $2$, $3$, and $5$.\\

Let $A$ be an abelian surface defined over an algebraically closed field $k$, and for the rest of this paper, let $G$ be a finite group acting on $A$ that preserves the group law, so that there is a monomorphism $G\hookrightarrow \End(A)^\times$. Let $\omega_A$ be a non-zero regular $2$-form on $A$. This is unique up to constant multiple as $h^{2,0}(A)=1$. Then, given an element $g$ of order $n$ of a group acting on $A$, we have that $g^*\omega_A$ must be equal to the product of $\omega_A$ by an $n$-th root of unity.

\begin{definition}
 The action of a finite group $G$ on an abelian variety $A$ is said to be \textbf{symplectic} if $g^*\omega_A=\omega_A$ for all $g\in G$.
\end{definition}

Note that in the case where $A$ is the product of two elliptic curves $E_1$ and $E_2$, $\omega_A$ can be expressed as $\omega_1\wedge \omega_2$ where $\omega_i$ is a global differential of $E_i$. 

\begin{definition}
The action of a finite group $G$ on an abelian variety $A$ is said to be \textbf{rigid} if and only if every non-trivial $g \in G$ fixes a finite, non-zero number of points.
\end{definition}

The main theorem in Katsura's paper states the following:
\begin{theorem}\cite[Theorem 2.4]{Katsura1987GeneralizedP} \label{Katsura_theorem}
Assume that the characteristic of the base field is not two. Then, a relatively minimal model of $A/G$ is a $K3$ surface if and only if $G$ satisfies the following conditions:
\begin{itemize}
\item The action is rigid.
\item The action is symplectic.
\item The quotient $A/G$ is singular, and all the singular points of $A/G$ are rational double points.
\end{itemize}
\end{theorem}
If the resolution of $A/G$ is a K3 surface, then Katsura proved in every characteristic that the three conditions have to be met. As for the converse statement, the proof relies on the fact that if the three conditions are satisfied, one can deduce that the canonical bundle of the minimal resolution of $A/G$ is trivial. Then, from the fact that the action is rigid, one can deduce that $A/G$ has singular points and therefore the resolution $S$ contains curves whose self-intersection number is negative. This implies that $S$ cannot be an abelian surface, a hyperelliptic surface or a quasi-hyperelliptic surface. If the characteristic of the base field is not two, we can deduce that $S$ is a K3 surface. If the characteristic is two, this is not necessarily true any more, but if we assume additional conditions, we can deduce that the quotient is a K3 surface.
\begin{proposition}
If we assume that the characteristic of the base field is two, then Theorem \ref{Katsura_theorem} still holds if the Picard number of the minimal resolution of $A/G$ is greater than $10$.
\end{proposition}
\begin{proof}
Kodaira's classification of surfaces states that if the characteristic of the base field is not two, then any minimal surface $S$ with trivial canonical bundle must be either abelian, hyperelliptic, quasi-hyperelliptic or K3. However, in characteristic two, there is the additional possibility that $S$ is a non-classical Enriques surface defined as the quotient of a K3 surface by the action of the group schemes $\mathbb{Z}/2\mathbb{Z}$ or $\alpha_2$ \cite[Section 7]{Liedtke2013AlgebraicCharacteristic}. Any condition that allows us to discard this case will suffice to deduce that the quotient is a K3 surface, for instance, if $\rho(S)>10$, then $S$ must be a K3 surface, as the second Betti number of an Enriques surface is $10$.
\end{proof}
Showing that the Picard number of the resolution is greater than $10$ is an easy condition to check from the singularities of $A/G$.\newpage In particular, in all the quotients that we considered in this paper the number of exceptional divisors that we obtain in the resolution of the singularities of $A/G$ will force the Picard number of any generalised Kummer surface to be at least $17$.\\

It is important to note that the three conditions of Theorem \ref{Katsura_theorem} are necessary. For instance, we have the following:

\begin{proposition}\cite[Theorem 2.11]{Katsura1987GeneralizedP} \label{ellsing_prop}
Let $G$ be a finite subgroup acting on $A$ such that the action of $G$ has no fixed curves. If $A/G$ has at least one singular point other than rational double points, then $A/G$ is a rational surface.
\end{proposition}

We will use this result in Section \ref{sec_examples} to prove the rationality of $A/G$ in some cases.

\subsection{Groups acting on abelian surfaces}
Some of the finite groups acting on surfaces in this context are non-standard, so it will be helpful to define them.

\begin{definition}
The \textbf{binary dihedral group} (or dicyclic group) $Q_{4n}$ of order $4n$ is the group defined by the presentation:
\begin{align*}
Q_{4n}=\langle a,b\mid a^nb^2,b^4, abab^{-1}\rangle.
\end{align*}
\end{definition}
These groups can also be characterised as being non-split extensions of $C_{2n}$ by a cyclic group of order two. The groups \href{https://beta.lmfdb.org/Groups/Abstract/8.4}{$Q_8$} and \href{https://beta.lmfdb.org/Groups/Abstract/16.9}{$Q_{16}$} are known as the \textbf{quaternion groups }of orders $8$ and $16$. As for the rest of them of order less than $24$, we have that \href{https://beta.lmfdb.org/Groups/Abstract/12.1}{$Q_{12}$}$\,\cong C_3\rtimes C_4$, \href{https://beta.lmfdb.org/Groups/Abstract/20.1}{$Q_{20}$}$\,\cong C_5\rtimes C_4$, and \href{https://beta.lmfdb.org/Groups/Abstract/24.4}{$Q_{24}$}$\,\cong C_3\rtimes Q_8$. Another relevant family of groups that we need to know is the following:

\begin{definition}
The \textbf{extended special linear group }$\ESL_2(\FF_{p})$ is the subgroup of
$\SL_2(\FF_{p^2})$ generated by $\SL_2(\FF_p)$ and an element given by the diagonal matrix 
\[M=\left (\begin{array}{c c}
\alpha & 0\\
0 & \alpha^{-1}\\	
\end{array}\right),\] where $\alpha\in\FF_{p^2}\setminus \FF_p$, and $\alpha^2$ is a primitive element which generates $\FF_p^\times$.
\end{definition}
\begin{remark}
Many of the groups acting on surfaces are \textbf{binary polyhedral groups}, that is, they are preimages under the map $\mathrm{SU}(2)\rightarrow\mathrm{SO}(3)$ of the group of rotations of a polyhedron. In addition to the binary dihedral groups, notable examples include: \href{https://beta.lmfdb.org/Groups/Abstract/24.3}{$\SL_2(\mathbb{F}_3)$}, which is isomorphic to the \textbf{binary tetrahedral group}; \href{https://beta.lmfdb.org/Groups/Abstract/48.28}{$\ESL_2(\mathbb{F}_3)$}, the \textbf{binary octahedral group}; and \href{https://beta.lmfdb.org/Groups/Abstract/120.5}{$\SL_2(\mathbb{F}_5)$}, the \textbf{binary icosahedral group}.
\end{remark}

Given a group $G$, the rational group algebra $\mathbb{Q}[G]$ is isomorphic to the sum of simple algebras $\mathbb{H}_V$, one for each irreducible representation $(V,\rho)$ over $\mathbb{Q}$. From this description, we can construct another algebra, the \textbf{rigid group algebra} $\mathbb{Q}[G]^{\rig}$ of $G$ as the sum over all the simple algebras corresponding to irreducible representations without fixed points,
\begin{align*}
\mathbb{Q}[G]^{\rig}=\bigoplus_{\substack{(V,\rho)\text{ without} \\ \text{fixed points}}}\mathbb{H}_V.
\end{align*}
Here, a representation is said to be \textbf{without fixed points} if $\rho(g)(v)\neq v$ for every nontrivial $g\in G$ and every nonzero $v\in V$.\\

Rybakov found a connection between the rigid group algebra of a group $G$ and rigid actions of $G$ on abelian varieties. More specifically, he proved the following:
\begin{proposition} \cite[Corollary 3.9]{Rybakov2024GeneralizedFields}
There exists an abelian variety with a rigid action of $G$ in the isogeny class of an abelian variety $A$, if and only if there exists a homomorphism of $\Q$-algebras $\Q[G]^{\rig}\rightarrow\End^\circ(A).$
\end{proposition}

Using a description of Dickson of the possible subgroups $G$ of $\SL_2(\bar{k})$ \cite{Dickson1901LinearTheory}, he then proved the following theorem:

\begin{theorem}[{\cite[Theorem 4.3]{Rybakov2024GeneralizedFields}}] \label{th_rigid_actions}
Let $G$ be a finite group with a rigid and symplectic action on an abelian surface over a field of characteristic $p\geq 0$.
\newpage Then $G$ is one of the following groups:
\begin{itemize}
\item $G$ is a cyclic group of order $n\in\{2,3,4,5,6,8,10,12\}$. Then, $\Q[C_n]^{\rig}=\Q(\zeta_n)$.\\

\item $G$ is a binary dihedral group $Q_{4n}$ of order $4n$, where $2\leq n\leq 6$. Then, $\Q[G]^{\rig}$ are
\begin{center}
\begin{tabular}{c|ccccc}
$G$  & $Q_8$ & $Q_{12}$ & $Q_{16}$                &  $Q_{20}$                  & $Q_{24}$ \\
\hline 
$\Q[G]^{\rig}$ &  $\bbH_{2,\infty}$  & $\bbH_{3,\infty}$ & $\bbH_\infty(\Q(\sqrt{2}))$ & $\bbH_\infty(\Q(\sqrt{5}))$ & $\bbH_\infty(\Q(\sqrt{3}))$ \\
\end{tabular}
\end{center}

\item $G$ is $\SL_2(\FF_3)$, $\ESL_2(\FF_3)$, or $\SL_2(\FF_5)$. Then, $\Q[G]^{\rig}$ are

\begin{center}
\begin{tabular}{c|ccc}
$G$  & $\SL_2(\FF_3)$             & $\ESL_2(\FF_3)$ & $\SL_2(\FF_5)$\\
\hline 
$\Q[G]^{\rig}$ & $\bbH_{2,\infty}$  &$\bbH_\infty(\Q(\sqrt{2}))$ & $\bbH_\infty(\Q(\sqrt{5}))$\\
\end{tabular}
\end{center}

\item If $p=2$, $G$ can be $C_3\rtimes C_8$, or $C_3\times Q_8$. Then, $\Q[G]^{\rig}$ are

\begin{center}
\begin{tabular}{c|cc}
$G$  &  $C_3\rtimes C_8$ & $C_3\times Q_8$   \\
\hline 
$\Q[G]^{\rig}$ & $M_2(\Q[\zeta_4])$ & $M_2(\Q[\zeta_3])$  \\
\end{tabular}
\end{center}
\vspace{10pt}
\item If $p=3$, $G$ can be $C_3\rtimes C_8$, $C_3\times Q_8$, or $C_3\rtimes Q_{16}$. Then, $\Q[G]^{\rig}$ are

\begin{center}
\begin{tabular}{c|ccc}
$G$  & $C_3\rtimes C_8$ & $C_3\times Q_8$ &  $C_3\rtimes Q_{16}$\\
\hline 
$\Q[G]^{\rig}$  & $M_2(\Q[\zeta_4])$ & $M_2(\Q[\zeta_3])$ &  $M_2(\bbH_{3,\infty})$  \\
\end{tabular}
\end{center}
\vspace{10pt}

\item If $p=5$, $G$ can be $\ESL_2(\FF_5)$, or $C_5\rtimes C_8$. Then, $\Q[G]^{\rig}$ are

\begin{center}
\begin{tabular}{c|ccc}
$G$  &             $\ESL_2(\FF_5)$ & $C_5\rtimes C_8$\\
\hline 
$\Q[G]^{\rig}$ & $M_2(\bbH_{5,\infty})$  & $M_2(\bbH_{5,\infty})$  
\end{tabular}
\end{center}
\vspace{10pt}

\end{itemize}
Here, $\bbH_{p,\infty}$ denotes the quaternion algebra over $\Q$ that splits at $p$ and at the place at $\infty$,
and $\bbH_\infty(K)$ is the quaternion algebra over a number field $K$ that splits at the real places of $K$. 
\end{theorem}

From Theorem \ref{Katsura_theorem}, Katsura produced a list of all finite groups $G$ such that $A/G$ can possibly be a generalised Kummer surface. Rybakov’s refined version of the classification states the following:

\begin{theorem}[{\cite[Proposition 5.1]{Rybakov2024GeneralizedFields}}]\label{classification_theorem}
Let $A$ be an abelian surface over a field of characteristic $p\geq0$ and assume $p\nmid |G|$. Then, $A/G$ is a generalised Kummer surface if and only if $G$ and $A$ satisfy the conditions listed in the first two columns of the table below, and the action of $G$ on $A$ is as in Theorem \ref{Katsura_theorem}.
\end{theorem}
\begin{adjustbox}{width=\textwidth}
\begin{minipage}{\textwidth}
\begin{table}[H]
\centering
\begin{tabular}{|c|cc|c|}
\hline
\rowcolor[HTML]{FFFFFF}
$G$                    & \multicolumn{2}{c|}{Conditions}                                              & Singularities of $A/G$         \\ \hline
$C_2$ & \multicolumn{2}{c|}{$p\neq2$}                                                & $16A_1$                \\ \hline
$C_3$ & \multicolumn{2}{c|}{$p\neq3$}                                                & $9A_2$                 \\ \hline
$C_4$ & \multicolumn{2}{c|}{$p\neq2$}                                                & $4A_3+6A_1$            \\ \hline
$C_5$                  & \multicolumn{1}{c|}{$p\equiv \pm 2\mod 5$}     & $f(A)=0$ & $5A_4$                 \\ \hline
$C_6$ & \multicolumn{2}{c|}{$p\neq2,3$}                                              & $A_5+4A_2+5A_1$        \\  \hline
$C_8$                  & \multicolumn{1}{c|}{$p\equiv \pm 3\mod 8$}     & $f(A)=0$ & $2A_7+A_3+3A_1$        \\ \hline
$C_{10}$               & \multicolumn{1}{c|}{$p\equiv \pm 2\mod 5$}     & $f(A)=0$ & $A_9+2A_4+3A_1$        \\ \hline
$C_{12}$               & \multicolumn{1}{c|}{$p\equiv \pm 5\mod 12$}    & $f(A)=0$ & $A_{11}+A_3+2A_2+2A_1$ \\ \hline
\multirow{2}{*}{$Q_8$} & \multicolumn{1}{c|}{\multirow{2}{*}{$p\neq2$}} & $\#Q_8\text{-fixed points}=2$                            & $2D_4+3A_3+2A_1$       \\ \cline{3-4} 
                       & \multicolumn{1}{c|}{}                          & $\#Q_8\text{-fixed points}=4$                              & $4D_4+3A_1$            \\ \hline
$Q_{12}$             & \multicolumn{2}{c|}{$p\neq2,3$}                                              & $D_5+3A_3+2A_2+A_1$    \\  \hline
$Q_{16}$               & \multicolumn{1}{c|}{$p\equiv \pm 3\mod 8$}     & $f(A)=0$ & $2D_6+D_4+A_3+A_1$     \\ \hline
$Q_{20}$               & \multicolumn{1}{c|}{$p\equiv \pm 2\mod 5$}     & $f(A)=0$ & $D_7+A_4+3A_3$         \\ \hline
$Q_{24}$               & \multicolumn{1}{c|}{$p\equiv \pm 5\mod12$}     & $f(A)=0$ & $D_8+D_4+2A_3+A_2$     \\ \hline
$\SL_2(\FF_3)$         & \multicolumn{2}{c|}{$p\neq2,3$}                                              & $E_6+D_4+4A_2+A_1$     \\  \hline
$\ESL_2(\FF_3)$        & \multicolumn{1}{c|}{$p\equiv \pm 3\mod 8$}     & $f(A)=0$ & $E_7+D_6+A_3+2A_2$     \\ \hline
$\SL_2(\FF_5)$         & \multicolumn{1}{c|}{$p\equiv \pm 2\mod 5$}     & $f(A)=0$ & $E_8+D_4+A_4+2A_2$     \\ \hline
\end{tabular}

\caption{The classification of generalised Kummer surfaces when $p\nmid |G|$.}
\label{table1}
\end{table}
\end{minipage}
\end{adjustbox}

The fact that neither of the order of the groups in Theorem \ref{th_rigid_actions} divides a prime $p\geq 7$ shows that the only cases left to prove are the ones in characteristics $p=2,3$ and $5$, when $p$ divides the order of $G$. In particular, the possible choices for $G$ and $p$ are 
\begin{itemize}
    \item If $p=2$, $G$ is $C_{2n}$ for $1\le n\le 6$, $Q_{4m}$ for $2\leq m\leq 6$, $\SL_2(\FF_3)$, $\ESL_2(\FF_3)$, $\SL_2(\FF_5)$,  $C_3\rtimes C_8$, or $C_3\times Q_8$.
    \item If $p=3$, $G$ is $C_3$, $C_6$, $Q_{12}$, $Q_{24}$, $\SL_2(\FF_3)$, $\ESL_2(\FF_3)$, $\SL_2(\FF_5)$, $C_3\rtimes C_8$, $C_3\times Q_8$, or $C_3\rtimes Q_{16}$.
    \item If $p=5$, $G$ is $C_5$, $C_{10}$, $\SL_2(\FF_5)$, $\ESL_2(\FF_5)$, or $C_5\rtimes C_8$.\\
\end{itemize}

We will now see how we can reduce the number of possible groups in this list by looking at the possible rational double points that the quotient variety can have.

\subsection{Rational double points in positive characteristic} \label{ss_rdp}
The main idea in this paper is that, if the quotient of a smooth surface by the group action of an étale group scheme $G$ only has rational double points, then this imposes some strong restrictions on what $G$ can be. To explain this phenomenon, we now recall the theory of rational double points, with particular emphasis on the differences between characteristic zero and positive characteristic.\\

There are several perspectives from which rational double points may be understood. From the point of view of their resolution, they are precisely the surface singularities whose exceptional divisors consist of $(-2)$-curves whose dual intersection graph is a simply laced Dynkin diagram of \textbf{ADE type}. On the other hand, if $P$ is a closed point of a normal surface $S$ defined over an algebraically closed field $k$, and $\widehat{\calO}_{P,S}$ denotes the completion of $\calO_{P,S}$, then, whenever $P$ is a rational double point, one can prove that there exists $f\in k[[x,y,z]]$ such that \begin{align*} \widehat{\calO}_{P,S}\cong k[[x,y,z]]/(f). \end{align*}

Over an algebraically closed field of characteristic zero, any two rational double points with the same desingularisation are locally analytically isomorphic. One can therefore describe rational double points locally by determining the possible choices of $f\in k[[x,y,z]]$. For each type of rational double point, $f$ may in fact be chosen to be a polynomial; these choices are referred to as \textbf{normal forms} \cite{Miyanishi2020AlgebraicCharacteristics}.\\

Artin showed that, in positive characteristic, the analogous statement fails in general: there may exist several rational double points with the same configuration of exceptional curves in their resolutions that are nevertheless not locally analytically isomorphic \cite{Artin1975CoveringsP}. However, for each ADE type, there are only finitely many such local isomorphism classes, which are usually distinguished by superscripts. For example, in every characteristic other than $2$ and $3$, an $E_6$ singularity is characterised by the normal form $f=z^2+x^3+y^4$. Nevertheless, in characteristic $3$ there are two non-isomorphic $E_6$ singularities: $E_6^0$, with normal form $f=z^2+x^3+y^4$, and $E_6^1$, with normal form $f=z^2+x^3+y^4+x^2y^2$.\\

Two rational double points with the same configuration of exceptional curves in their resolution may be distinguished by their \textbf{Tjurina number}, which is defined by
\begin{align*}
\tau(P)=\dim_k \frac{k[[x,y,z]]}{\langle f,\; \tfrac{\partial f}{\partial x},\; \tfrac{\partial f}{\partial y},\; \tfrac{\partial f}{\partial z}\rangle}.
\end{align*}
For instance, in characteristic three the $E_6^0$ singularity has Tjurina number 6, whereas the $E_6^1$ has Tjurina number 7.\\

Finally, another way of characterising rational double points is in terms of their local fundamental groups.
\begin{definition}
Let $\calO_{P,S}^h$ be the henselisation of $\calO_{P,S}$ and $X=\Spec \calO^h_{P,S}$. A \textbf{covering of $\boldsymbol{X}$} is a finite surjective map $Y\rightarrow X$ such that $Y$ is irreducible and normal. The \textbf{local fundamental group $\pi_1(X\setminus P)$} is the group which classifies finite coverings of $X$ that are étale except above $P$.
\end{definition}

Over the complex numbers, one way of characterising rational double points is to describe them as the singularities which are, locally, the quotient of the affine plane by the action of a finite group  $G<\SL_2(\mathbb{C})$. This can be shown to be equivalent to showing that the local fundamental group in a neighbourhood of the singularity is $G$ \cite{Prill1967LocalGroups}.\\

This characterisation is no longer true in positive characteristic, as there are rational double points whose local fundamental group is trivial. For instance, this is the case for the local fundamental group of a singularity of type $A_{p-1}$ in characteristic $p$. These singularities are therefore not locally the quotients of the affine plane by étale group schemes. Instead, they can be understood as quotients by the infinitesimal group scheme $\mu_p$. This perspective is further elaborated in Miyanishi's and Ito's book \cite{Miyanishi2020AlgebraicCharacteristics}.\\

In 1975, Artin classified all rational double points in positive characteristic in terms of their desingularisation, normal form, local fundamental group and Tjurina numbers \cite{Artin1975CoveringsP}. The results of this paper can be summarised in the following tables, where $G'$ denotes the maximal prime-to-$p$ quotient of $G$ and $D_{(2r-n)'}$ is the dihedral group of order $2m$ where $m$ is the greatest divisor of $(2r-n)$ which is not divisible by two.

\begin{table}[H]
\centering
\begin{tabular}{|cl|c|c|c|}
\hline
\multicolumn{2}{|c|}{\textbf{Type} } & \textbf{Normal form} & $\boldsymbol{\pi_1(X\setminus P)}$ & $\boldsymbol{\tau(P)}$ \\
\hline
\multicolumn{1}{|c}{$A_n$} &   $(n\text{ even})$ & $xy+z^{n+1} $ & $(C_n)'$ & $n$ \\\hline
\multicolumn{1}{|c}{$A_n$} &  $(n\text{ odd})$ & $xy+ z^{n+1}$ & $C_n$ & $n + 1$ \\\hline
\multicolumn{1}{|c}{$D_{2n}^0$} & $(n \geq 2)$ & $z^2 + x^2 y + x y^n$ & $0$ & $4n$ \\
\multicolumn{1}{|c}{$D_{2n}^r$} & $ (n \geq 2, 1 \leq r < \tfrac{n}{2})$ & $z^2 + x^2 y + x y^n + x y^{n - r} z$ & $0$ & $4n - 2r$ \\
\multicolumn{1}{|c}{$D_{2n}^r$} & $ (n \geq 2, r=\tfrac{n}{2})$ & $z^2 + x^2 y + x y^n + x y^{n - r} z$ & $C_2$ & $4n - 2r$ \\
\multicolumn{1}{|c}{$D_{2n}^r$} & $(n \geq 2, \tfrac{n}{2}< r \leq n-1)$ & $z^2 + x^2 y + x y^n + x y^{n - r} z$ & $D_{(2r-n)'}$ & $4n - 2r$ \\\hline
\multicolumn{1}{|c}{$D_{2n+1}^0$} & $(n \geq 2)$ & $z^2 + x^2 y + y^n z$ & $0$ & $4n$ \\
\multicolumn{1}{|c}{$D_{2n+1}^r$} & $(n \geq 2, 1 \leq r < \tfrac{n}{2})$ & $z^2 + x^2 y + y^n z + x y^{n - r} z$ & $0$ & $4n - 2r$ \\
\multicolumn{1}{|c}{$D_{2n+1}^r$} & $(n \geq 2, \tfrac{n}{2} < r \leq n-1)$ & $z^2 + x^2 y + y^n z + x y^{n - r} z$ & $D_{4r-2n+1}$ & $4n - 2r$ \\
\hline
\multicolumn{2}{|c|}{$E_6^0$} & $z^2 + x^3 + y^2 z$ & $C_3$ & $8$ \\
\multicolumn{2}{|c|}{$E_6^1$} & $z^2 + x^3 + y^2 z + xyz$ & $C_6$ & $6$ \\
\hline
\multicolumn{2}{|c|}{$E_7^0$} & $z^2 + x^3 + x y^3$ & $0$ & $14$ \\
\multicolumn{2}{|c|}{$E_7^1$} & $z^2 + x^3 + x y^3 + y^4 z$ & $0$ & $12$ \\
\multicolumn{2}{|c|}{$E_7^2$} & $z^2 + x^3 + x y^3 + x^2 y z$ & $0$ & $10$ \\
\multicolumn{2}{|c|}{$E_7^3$} & $z^2 + x^3 + x y^3 + xyz$ & $C_4$ & $8$ \\
\hline
\multicolumn{2}{|c|}{$E_8^0$} & $z^2 + x^3 + y^5$ & $0$ & $16$ \\
\multicolumn{2}{|c|}{$E_8^1$} & $z^2 + x^3 + y^5 + xyz$ & $0$ & $14$ \\
\multicolumn{2}{|c|}{$E_8^2$} & $z^2 + x^3 + y^5 + y^3 z$ & $C_2$ & $12$ \\
\multicolumn{2}{|c|}{$E_8^3$} & $z^2 + x^3 + y^5 + y^2 z$ & $0$ & $10$ \\
\multicolumn{2}{|c|}{$E_8^4$} & $z^2 + x^3 + y^5 + xyz$ & $Q_{12}$ & $8$ \\
\hline
\end{tabular}

\caption{Possible rational double points in characteristic two.}
\end{table}

\begin{table}[h!]
\centering
\begin{tabular}{|cl|c|c|c|}
\hline
\multicolumn{2}{|c|}{\textbf{Type}} & \textbf{Normal form} & $\boldsymbol{\pi_1(X\setminus P)}$ & $\boldsymbol{\tau(P)}$ \\
\hline
\multicolumn{1}{|c}{$A_n$} & $(3 \nmid (n+1))$ & $z^2 + x^2 + y^{n+1}$ & $C_{n+1}$ & $n$ \\\hline
\multicolumn{1}{|c}{$A_n$} & $\left(3 \mid (n+1)\right)$ & $z^2 + x^2 + y^{n+1}$ & $(C_{n+1})'$ & $n + 1$ \\
\hline
\multicolumn{1}{|c}{$D_n$} & $ (n \geq 4)$ & $z^2 + x^2 y + y^{n-1}$ & $(Q_{4n-8})'$ & $n$ \\
\hline
\multicolumn{2}{|c|}{$E_6^0$} & $z^2 + x^3 + y^4$ & $0$ & $6$ \\
\multicolumn{2}{|c|}{$E_6^1$} & $z^2 + x^3 + y^4 + x^2 y^2$ & $C_3$ & $7$ \\
\hline
\multicolumn{2}{|c|}{$E_7^0$} & $z^2 + x^3 + x y^3$ & $C_2$ & $9$ \\
\multicolumn{2}{|c|}{$E_7^1$} & $z^2 + x^3 + x y^3 + x^2 y^2$ & $C_6$ & $7$ \\
\hline
\multicolumn{2}{|c|}{$E_8^0$} & $z^2 + x^3 + y^5$ & $0$ & $12$ \\
\multicolumn{2}{|c|}{$E_8^1$} & $z^2 + x^3 + y^5 + x^2 y^3$ & $0$ & $10$ \\
\multicolumn{2}{|c|}{$E_8^2$} & $z^2 + x^3 + y^5 + x^2 y^2$ & $\SL_2(\mathbb{F}_3)$ & $8$ \\
\hline
\end{tabular}

\caption{Possible rational double points in characteristic three.}
\end{table}

\begin{table}[h!]
\centering
\begin{tabular}{|cl|c|c|c|}
\hline
\multicolumn{2}{|c|}{\textbf{Type}} & \textbf{Normal form} & $\boldsymbol{\pi_1(X\setminus P)}$ & $\boldsymbol{\tau(P)}$ \\
\hline
\multicolumn{1}{|c}{$A_n$} & $\left(5 \nmid (n+1)\right)$ & $z^2 + x^2 + y^{n+1}$ & $C_{n+1}$ & $n$ \\\hline
\multicolumn{1}{|c}{$A_n$} & $\left(5 \mid (n+1)\right)$ & $z^2 + x^2 + y^{n+1}$ & $(C_{n+1})'$ & $n + 1$ \\
\hline
\multicolumn{1}{|c}{$D_n$} & $(n \geq 4)$ & $z^2 + x^2 y + y^{n-1}$ & $(Q_{4n-8})'$ & $n$ \\
\hline
\multicolumn{2}{|c|}{$E_6$} & $z^2 + x^3 + y^4$ & $\SL_2(\mathbb{F}_3)$ & $6$ \\
\hline
\multicolumn{2}{|c|}{$E_7$} & $z^2 + x^3 + x y^3$ & $\ESL_2(\mathbb{F}_3)$ & $7$ \\
\hline
\multicolumn{2}{|c|}{$E_8^0$} & $z^2 + x^3 + y^5$ & $0$ & $10$ \\
\multicolumn{2}{|c|}{$E_8^1$} & $z^2 + x^3 + y^5 + x y^4$ & $C_5$ & $8$ \\
\hline
\end{tabular}
\caption{Possible rational double points in characteristic five.}
\end{table}
\vspace{5pt}

\begin{table}[h!]
\centering
\begin{tabular}{|cl|c|c|c|}
\hline
\multicolumn{2}{|c|}{\textbf{Type}} & \textbf{Normal form} & $\boldsymbol{\pi_1(X\setminus P)}$ & $\boldsymbol{\tau(P)}$ \\
\hline
\multicolumn{1}{|c}{$A_n$} & $(p \nmid (n+1))$ & $z^2 + x^2 + y^{n+1}$ & $C_{n+1}$ & $n$ \\\hline
\multicolumn{1}{|c}{$A_n$} & $(p \mid (n+1))$ & $z^2 + x^2 + y^{n+1}$ & $(C_{n+1})'$ & $n + 1$ \\
\hline
\multicolumn{1}{|c}{$D_n$} & $(n \geq 4)$ & $z^2 + x^2 y + y^{n-1}$ & $(Q_{4n-8})'$ & $n$ \\
\hline
\multicolumn{2}{|c|}{$E_6$} & $z^2 + x^3 + y^4$ & $\SL_2(\mathbb{F}_3)$ & $6$ \\
\hline
\multicolumn{2}{|c|}{$E_7$} & $z^2 + x^3 + x y^3$ & $\ESL_2(\mathbb{F}_3)$ & $7$ \\
\hline
\multicolumn{2}{|c|}{$E_8$} & $z^2 + x^3 + y^5$ & $\SL_2(\mathbb{F}_5)$ & $8$ \\
\hline
\end{tabular}

\caption{Possible rational double points in characteristic $p\geq7$.}
\end{table}
\newpage
The following result will allow us to use the information in these tables to study the singularities of quotients.
\begin{lemma} \label{prop_fund}
Let $G$ be a finite group acting rigidly on an abelian surface $A$ and let $\varphi: A\rightarrow A/G$ be the quotient map. Let $P$ be a point of $A$ such that $\Stab_G(P)$ is a normal subgroup of $G$ and let $U=\Spec(\mathcal{O}^h_{\varphi(P),A/G})$. Then, $\Stab_G(P)$ is isomorphic to a subgroup of $\pi_1(U\setminus \varphi(P))$.
\end{lemma}
\begin{proof}
The map $\varphi: A \rightarrow A/G$ is finite, and since the action is rigid, it is étale away from the fixed points of the action of $G$. Now, as $\Stab_G(P)$ is a normal subgroup of $G$, the quotient $\varphi$ factors as
\begin{align*}
A \xrightarrow{\phi} A/\Stab_G(P) \xrightarrow{\psi} A/G
\end{align*}
where $\psi$ is the quotient of $A/\Stab_G(P)$ by the induced action of the group $G/\Stab_G(P)$.\\

 Therefore, we have the sequence of ring homomorphisms
 \begin{align*}
 \calO_{\pi(P),A/G}\rightarrow \calO_{\phi(P),A/\Stab_G(P)}\rightarrow\calO_{P,A}
 \end{align*}
 and applying the universal property of henselisation and taking spectra, we obtain
 \begin{align*}
 \Spec (\calO_{P,A}^h) \rightarrow \Spec (\calO^h_{\phi(P),A/\Stab_G(P)})\rightarrow \Spec (\calO_{\varphi(P),A/G}^h)\\[-5pt]
 \end{align*}
 Now, as the map $\psi$ is given by the quotient of $A/\Stab_G(Q)$ by the group action induced on $A$ by $G$, we deduce that $G$ acts freely in a neighbourhood of $\phi(P)$ and the map 
 \begin{align*}
\Spec (\calO^h_{\phi(P),A/\Stab_G(P)})\rightarrow \Spec (\calO_{\varphi(P),A/G}^h)
 \end{align*} is étale.
 As a consequence, 
\begin{align*}
\pi_1(U\setminus\varphi(P))=\pi_1(\Spec (\calO^h_{\phi(P),A/\Stab_G(P)})\setminus\phi(P)).\\[-5pt]
\end{align*}
Every element $g\in\Stab_G(P)$ fixes $\phi(P)$ and therefore induces a finite covering of $U$ that is étale except above $\phi(P)$. Thus, we conclude that 
\begin{equation*}
\Stab_G(P)\leq \pi_1(\Spec (\calO^h_{\phi(P),A/\Stab_G(P)})\setminus\phi(P))=\pi_1(U\setminus \varphi(P)). \qedhere
\end{equation*}
\end{proof}

\newpage
Combining the classification of groups in Theorem \ref{th_rigid_actions} and the local information from Lemma \ref{prop_fund}, we obtain the following result:

\begin{proposition} \label{primegroups_propo}
    Let $A$ be an abelian surface over a field of characteristic $p$, and let $G$ be a finite group acting rigidly and symplectically on $A$ such that $p \mid |G|$ and the singularities of $A/G$ are all rational double points. Then, \vspace{0pt}
\begin{itemize}
    \item If $p=2$, $G$ can only be $C_2$, $C_4$, $C_6$ or $Q_{12}$.
    \item If $p=3$, $G$ can only be $C_3$, $C_6$ or $\SL_2(\mathbb{F}_3)$.
    \item If $p=5$, $G$ can only be $C_5$.
\end{itemize}
\end{proposition}

\begin{proof}
Since $G \hookrightarrow \End(A)^\times$, every element $g \in G$ fixes the identity element $O$ of $A$. By Lemma \ref{prop_fund}, the local fundamental group of a punctured neighbourhood around the image of $O$ in $A/G$ must contain $G$ as a subgroup. Thus, $G$ must be a subgroup of one of the local fundamental groups listed in the tables of Subsection \ref{ss_rdp} for rational double point singularities in characteristic $p$.\\

In characteristic $2$, the only local fundamental groups whose orders are divisible by two are $C_2$, $C_4$, $C_6$, $Q_{12}$ and $D_m$ for $m$ odd. The only groups among those listed in Theorem \ref{th_rigid_actions} that are subgroups of any of these are precisely $C_2$, $C_4$, $C_6$ and $Q_{12}$.\\

In characteristic $3$, the only possible local fundamental groups of order divisible by $3$ are $C_3$, $C_6$, and $\SL_2(\mathbb{F}_3)$, all of which appear in Theorem \ref{th_rigid_actions}.\\

In characteristic $5$, the only possibility consistent with Theorem~\ref{th_rigid_actions} is $C_5$. Furthermore, in this case, Rybakov has shown that $A$ must be supersingular \cite[Corollary 5.4]{Rybakov2024GeneralizedFields}.
\end{proof}

Now we have a list of possible groups that can act on abelian surfaces; let us further analyse which possibilities can happen and what the possible singularities of the quotients are.\\

\section{The singularities of the quotient} \label{computingsings_section}

A useful point of view to study an action of a group $G$ on an abelian variety $A$ is to study the induced action of $G$ on the $\ell$-adic Tate module of $A$. If $p\nmid \lvert G\rvert$, for instance, we have the following result:
\begin{proposition}[{\cite[Proposition 3.3]{Rybakov2024GeneralizedFields}}] \label{rigid_propo}
Let $G$ act on an abelian variety $A$ over a field of characteristic $p$, and let $\ell \neq p$ and $p\nmid \lvert G\rvert$. The following are equivalent: 
\begin{itemize}
    \item The action of $G$ on $A$ is \textbf{rigid} which, as we saw, meant that every non-trivial element $g \in G$ has only finitely many fixed points.
    \item The representation of $\,G$ in $V_\ell(A)$ is \textbf{without fixed points}, i.e., for any $g\in G$ of order $n$ the eigenvalues of the action of $g$ on $V_\ell(A)$ are primitive $n$-th roots of unity.
\end{itemize}
\end{proposition}
We would like to establish a similar result in the case $\ell=p$. However, this is not possible in complete generality: for instance, if $A$ is supersingular, then $V_p(A)$ is trivial. As we will see in the next section, if $A$ is a supersingular abelian surface and $p\mid |G|$, then the quotient $A/G$ can never be a generalised Kummer surface. For this reason, throughout this section we assume that $A$ is an abelian surface over a field of characteristic $p$ that is not supersingular. For abelian surfaces, this is equivalent to requiring that the $p$-rank $f(A)$ be nonzero. Under this assumption, the $p$-adic Tate module $T_p(A)$ is a free $\mathbb{Z}_p$-module of rank $f(A)$. In this setting, we can still obtain an analogue of Proposition \ref{rigid_propo}:

\begin{proposition}
Let $A$ be an abelian surface in characteristic $p$ that is not supersingular, and suppose $G$ acts rigidly on $A$. Then the representation of $G$ on $V_p(A)$ is without fixed points.
\end{proposition}

\begin{proof}
Suppose that the representation of $G$ on $V_p(A)$ was not fixed-point free. Then there would be an element $g \in G$ of order $n$ and a primitive $r$-th root of unity $\zeta_r$ with $r < n$ such that $\zeta_r$ is an eigenvalue of $g$. In this case, the action of $g^r$ on $V_p(A)$ would have $1$ as an eigenvalue, implying that $g^r$ fixes a non-zero vector $v \in T_p(A)$. But then, for each $j \geq 1$, the torsion point $v_j \in A[p^j]$ corresponding to the compatible system defining $v$ would also be fixed by $g^r$. This shows that there would be infinitely many fixed points, and it leads to a contradiction.
\end{proof}

Let $g\in G$ be an element of order $n$ acting rigidly. Whenever $p\nmid n$ or $f(A)>0$, from the above propositions we deduce that the representation of $G$ in $V_p(A)$ is without fixed points. Therefore, the Tate module $T_p(A)$ is a free module over the discrete valuation ring $\Z_p[\zeta_n]$, where $\zeta_n$ acts as $g$. Therefore, we can generalise Lemma 5.2 of Rybakov's paper to the following result:

\begin{proposition}\label{prime_sing_prop}
Let $G$ be a finite group acting rigidly on an abelian variety $A$ and let $\ell$ be a prime. If $\ell=p$, assume that $f(A)>0$.
\begin{enumerate}
        \item If $P\in A$ is a point fixed by an element $g \in G$ of order $\ell^r$, then $P\in A[\ell](\bar{k})$.
        \item If $G=C_{\ell^r}$, the set of fixed points of $G$ 
        is a subgroup of $A[\ell](\bar{k})$ of order $\ell^{m}$, where \begin{align*}
        m=\frac{\log_\ell(\#A[\ell])}{(\ell-1)\ell^{r-1}}.\\
        \end{align*}
       \end{enumerate}
\end{proposition}
\begin{proof}
\begin{enumerate}
\item 
As $g \in G$ has order $\ell^r$, $T_\ell(A)$ is a free module over $\Z_\ell[\zeta_{\ell^r}]$ and if $P$ is a fixed point of $g$, then $(g-\mathrm{id})(P)=O$. Therefore, the norm $\mathrm{Norm}_{\Q(\zeta_{\ell^r})/\Q}(\zeta_{\ell^r}-1)=\ell$ also annihilates $P$, so $P\in A[\ell](\bar{k})$.\\

\item Let $m$ be the rank of $T_\ell(A)$ as a free module over $\Z_\ell[\zeta_{\ell^r}]$. We know that, as $\Z_\ell$-module, $\rank_{\Z_\ell}(T_\ell(A))=\log_\ell(\#A[\ell])$, which is $4$ if $l\neq p$ and $f(A)$ if $\ell=p$. Furthermore,
\begin{align*}
\Z_\ell[\zeta_{\ell^r}]=\Z_\ell[x]/(\Phi_{\ell^r}(x)\Z_\ell[x])
\end{align*}
where $\Phi_{\ell^r}(x)$ is the $\ell^r$-th cyclotomic polynomial. As the degree of $\Phi_{\ell^r}(x)$ is $\varphi(\ell^r)=(\ell-1)\ell^{r-1}$, we deduce that
\begin{align*}
m=\frac{\log_\ell(\#A[\ell])}{(\ell-1)\ell^{r-1}}.\\[-5pt]
\end{align*}

Now, the set of fixed points of $G$ is isomorphic to
\begin{align*}
T_\ell(A)/((g-\mathrm{id})T_\ell(A))\cong (\Z_\ell[\zeta_{\ell^r}]/((\zeta_{\ell^r}-1)\Z_\ell[\zeta_{\ell^r}]))^{m}.\\[-5pt]
\end{align*}
We know that $\Z_\ell[\zeta_{\ell^r}]$ is a completely ramified extension of $\Z_\ell$ and one can check that $(\zeta_{\ell^r}-1)$ is a uniformiser of $\Z_\ell[\zeta_{\ell^r}]$, so
\begin{align*}
\Z_\ell[\zeta_{\ell^r}]/((\zeta_{\ell^r}-1)\Z_\ell[\zeta_{\ell^r}])\cong\mathbb{F}_\ell
\end{align*}
and, therefore, the number of fixed points is $\ell^{m}$. \qedhere
\end{enumerate}
\end{proof}

One thing that is hinted in this proposition is that for an element $g$ of order $\ell^r$ to act rigidly on $A$, $(\ell-1)\ell^{r-1}$ must divide $\log_\ell(\#A[\ell])$. This is true because if $G=\langle g\rangle$ acts rigidly, the action of $G$ on $V_\ell(A)$ has to be semisimple with eigenvalues given by primitive $\ell^r$-th roots of unity.\\

As a consequence, the representation of $G$ on $V_\ell(A)$ factors through a direct sum of characters that are $\ell^r$-th roots of unity. Not all of these characters are defined over $\mathbb{Q}_\ell$, but they become defined over $\mathbb{Q}_\ell(\zeta_{\ell^r})$, where $\zeta_{\ell^r}$ is a primitive $\ell^r$-th root of unity. The Galois group \( \mathrm{Gal}(\mathbb{Q}_\ell(\zeta_{\ell^r})/\mathbb{Q}_\ell) \) acts on these characters, and each primitive character of order \( \ell^r \) has a full Galois orbit of size $[\mathbb{Q}_\ell(\zeta_{\ell^r}) : \mathbb{Q}_\ell] = \varphi(\ell^r)$.
Each such orbit spans an irreducible representation of \( G \) over \( \mathbb{Q}_\ell \) of dimension \( \varphi(\ell^r) \), corresponding to the sum of all Galois conjugates of a given primitive character. Therefore, to write down a rigid action of \( G \) on \( V_\ell(A) \), we must be able to fit copies of these irreducible representations of dimension \( \varphi(\ell^r) \) into \( V_\ell(A) \), which has dimension $\log_\ell(\#A[\ell])$. Hence, we obtain the divisibility condition:
\begin{align*}
(\ell - 1)\ell^{r - 1} \mid \log_\ell(\#A[\ell]).\\
\end{align*}

For abelian surfaces, when $\ell\neq p$, $\log_\ell(\#A[\ell])=4$ and we therefore deduce that the possibilities are $(\ell,r)\in\{(2,1),(2,2),(2,3),(3,1),(5,1)\}$.\\

If $\ell= p$, $\log_\ell(\#A[\ell])=f(A)$, so if $f(A)=2$ we deduce that the only possibilities that can happen are $(\ell,r)\in\{(2,1),(2,2),(3,1)\}$, and if $f(A)=1$ the only possibility is $(\ell,r)=(2,1)$.\\

If $G$ is not cyclic, we can use the information about the Sylow $\ell$-subgroups of $G$ to deduce information about the fixed points of $G$:

\begin{proposition}[{\cite[Lemma 5.2]{Rybakov2024GeneralizedFields}}] \label{comp_sing_prop}
Under the same assumptions as in Proposition \ref{prime_sing_prop}, we also have the following: 
\begin{enumerate}
\item If $H$ is a subgroup of $G$ such that $\ell$ does not divide the order of $H$, then the action of $H$ on $A[\ell](\bar{k})$ only fixes $O$.
\item Let $G^{(\ell)}$ be a fixed Sylow $\ell$-subgroup of $G$, and let $n_\ell$ be the number of points in $A$ whose stabiliser is exactly $G^{(\ell)}$. Let $s_\ell$ be the number of conjugacy classes of Sylow $\ell$-subgroups of $G$. Then the total number of points that are fixed by any group Sylow $\ell$-subgroup is $s_\ell n_\ell$.
\end{enumerate}
\end{proposition}
Combining the results of the Propositions \ref{prime_sing_prop} and \ref{comp_sing_prop}, we deduce that if $p$ divides the order of $G$ and $f(A) > 0$, the only possible groups $G$ acting rigidly and symplectically on an abelian surface $A$ and values of $f(A)$ such that $A/G$ has only rational double points are the following:

\begin{adjustbox}{width=\textwidth}
\begin{minipage}{\textwidth}
\begin{table}[H]
\centering
\begin{tabular}{|c|cc|c|}
\hline
\rowcolor[HTML]{FFFFFF}
$G$                    & \multicolumn{2}{c|}{Conditions}                                              & Singularities of $A/G$         \\ \hline
\multirow{2}{*}{$C_2$}
                       & \multicolumn{1}{c|}{\multirow{2}{*}{$p=2$}}    & $f(A)=2$         & $4 D_4^1$              \\ \cline{3-4} 
                       & \multicolumn{1}{c|}{}                          & $f(A)=1$         & $2 D_8^2$              \\ \hline
$C_3$                       & \multicolumn{1}{c|}{$p=3$}                     & $f(A)=2$     & $3E_6^1$               \\ \hline
$C_4$                        & \multicolumn{1}{c|}{$p=2$}                     & $f(A)=2$       & $2E_7^3+D_4^1$          \\ \hline
\multirow{2}{*}{$C_6$}                        & \multicolumn{1}{c|}{$p=2$}                     & $f(A)=2$       & $E_6^1+D_4^1+4 A_2$    \\ \cline{2-4} 
                       & \multicolumn{1}{c|}{$p=3$}                     & $f(A)=2$       & $E_7^1+E_6^1+5A_1$     \\ \hline
$Q_{12}$              & \multicolumn{1}{c|}{$p=2$}                     & $f(A)=2$       & $E_8^4+E_7^3+2A_2$          \\ \hline
$\SL_2(\FF_3)$                        & \multicolumn{1}{c|}{$p=3$}                     & $f(A)=2$       & $E_8^2+E_6^1+D_4+A_1$               \\ \hline
\end{tabular}
\caption{A classification of generalised Kummer surfaces when $p\mid |G|$ and $f(A)>0$.}
\label{classp_table}
\end{table}
\end{minipage}
\end{adjustbox}\vspace{20pt}
In this table, the number and ADE types of the singularities have been computed using Proposition \ref{comp_sing_prop}, following the same method as in the proof of Proposition 5.1 in Rybakov’s paper. In cases where several singularities share the same ADE type, we determined their local isomorphism classes using the information on local fundamental groups contained in the tables of Section \ref{ss_rdp}.\\

In Section \ref{sec_examples}, we will see that all of the possibilities in this table can indeed occur even when we restrict to the case in which $A$ is a product of two elliptic curves. Before turning to those examples, however, let us first deal with the supersingular case.\\

\section{About the supersingular case} \label{ss_section}
The goal of this section is to show that if $A$ is a supersingular abelian surface and  $p \mid |G|$, then, the desingularisation of $A/G$ cannot possibly be a K3 surface. 

\subsection{Reducing to the case where \texorpdfstring{$\boldsymbol{G=C_p}$}{G=Cp}}
The following result shows that, it is enough to prove the statement for the case where $G=C_p$.

\begin{proposition} \label{quot_prop}
Let $A$ be a supersingular abelian surface defined over a field of characteristic $p$. Assume that whenever $p\in\{2,3,5\}$,  $A/C_p$ is never a generalised Kummer surface. Then, for any group $G$ such that $p\mid |G|$, $A/G$ is never a generalised Kummer surface.
\end{proposition}

\begin{proof}
We have seen that if $p \mid |G|$, $A/G$ can only be a generalised Kummer surface if $p$ is $2,3$ or $5$ and $G$ is one of the groups described in Proposition \ref{primegroups_propo}.\\

Since $C_2$ is a normal subgroup of $C_4$, $C_6$ and $Q_{12}$, and $C_3$ is a normal subgroup of $C_6$ and $\SL_2(\mathbb{F}_3)$, in each of these cases the quotient map $A\rightarrow A/G$ factors through the quotient $A\rightarrow A/C_p$. Suppose that $A/G$ were a generalised Kummer surface. Then, from Theorem \ref{Katsura_theorem}, we would deduce that the singularities of $A/G$ would all be rational double points, and from the fact that the quotient maps are finite, we deduce that $A/C_p$ would have isolated singularities.\\

On the other hand, since $A/C_p$ is not a generalised Kummer surface, we would deduce that $A/C_p$ should have at least one singular point that is not a rational double point. By Proposition \ref{ellsing_prop}, this would imply that that the minimal desingularisation of $A/C_p$ is a rational surface. This is impossible, however, as it would imply that there is a finite dominant morphism from a rational surface into the minimal desingularisation of $A/G$, which is a K3 surface.
\end{proof}

\subsection{Characteristic two}

The main hypothesis of Proposition \ref{quot_prop} is satisfied in characteristic two because of a fact that holds in all characteristics: whenever $A/C_2$ is a generalised Kummer surface, it is necessarily a Kummer surface. 
\begin{proposition} \label{prop_order2}
Let $G=C_2$ act rigidly on an abelian variety $A$. Then, $C_2=\langle\iota\rangle$ where $\iota:A\rightarrow A$ is the map that sends any point to its inverse.
\end{proposition}
\begin{proof}
Let $g\in G$ be element of $G$ of order two and let $\psi \in \End(A)$ be $\psi = g - \mathrm{id}$. A point $P \in A$ is fixed by $g$ if and only if $g(P) = P$, which is equivalent to $\psi(P) = O$. Therefore, the fixed locus of $g$ is exactly $\ker(\psi)$. As $G$ acts rigidly on $A$, we deduce that $\psi$ is an isogeny and therefore, $\psi\in \End^0(A) = \End(A) \otimes_{\mathbb{Z}} \mathbb{Q}$.\\

As $g^2=\mathrm{id}$, we deduce that
\begin{align*}
0 = g^2 - \mathrm{id}=(g-\mathrm{id})(g+\mathrm{id})=\psi(g+\mathrm{id})
\end{align*}
and, therefore, as $\psi\in \End^0(A)$, composing with $\psi^{-1}$ we deduce that $g+\mathrm{id}=0$, and therefore, $g=\iota$.
\end{proof}

\begin{corollary} \label{two_cor}
Let $A$ be a supersingular abelian surface in characteristic two and let $C_2$ act on $A$. Then, $A/C_2$ is not a generalised Kummer surface.
\end{corollary}
\begin{proof}
From Proposition \ref{prop_order2}, we deduce that the only possible quotient of a supersingular abelian surface that acts rigidly on an abelian variety in characteristic two by an action of order two is the Kummer surface. 
From the work of Katsura on Kummer surfaces in characteristic two \cite{Katsura1978On2}, the quotient of a supersingular abelian surface by the action of the sign involution has elliptic singularities, and therefore, from Proposition \ref{ellsing_prop}, we deduce that the desingularisation is always a rational surface.
\end{proof}
The equations for the model of the quotient surface $A/\iota$ in the case where $A$ is the product of two elliptic curves can be found in Subsection \ref{order2_subsection}, whereas the case where $A$ is the Jacobian of a genus two curve can be found in an earlier paper by the author \cite{Gonzalez-Hernandez2026ExplicitSpecialisation}.
\begin{remark}
In light of Corollary \ref{two_cor}, it is worth mentioning Matsumoto’s work on \textbf{inseparable Kummer surfaces} \cite{Matsumoto2024InseparableSurfaces}. He showed that, although the quotient of a supersingular abelian surface by $\langle\iota\rangle$ is not a K3 surface, one can construct inseparable analogues that share many key properties with classical Kummer surfaces. For example, these inseparable Kummer surfaces admit coverings $Y \rightarrow X$ such that the smooth locus of $Y$ is a group variety.
\end{remark}

\subsection{Characteristics three and five}
\begin{proposition} \label{fixorig_prop}
Let $G=C_p$ act rigidly on a supersingular abelian variety in characteristic $p$. Then, $G$ only fixes the origin.
\end{proposition}
\begin{proof}
Similarly to the proof of Proposition $\ref{prop_order2}$, let $g\in G$ and $\psi=g-\mathrm{id}$. Then, as $G$ acts rigidly, $\psi\in\End^0(A)$ and from the fact that $g^p=\mathrm{id}$, we deduce that
$\Phi_p(g)=0$, where $\Phi_p$ is the $p$-th cyclotomic polynomial. Let $P\in A$ be a fixed point of $g\in G$. 
Then, we have that 
\begin{align*}
[p]P=P+P+\dots+P=g^{p-1}(P)+g^{p-2}(P)+\dots+g(P)+P=\Phi_p(g)(P)=O
\end{align*}
from which we deduce that $P\in A[p]$. As $A$ is supersingular, we conclude that $P=O$.
\end{proof}
As a consequence of this proposition, we deduce that $A/C_p$ can only possibly have one singular point.

\begin{lemma}[{\cite[Proposition 3.2]{Harder1975OnCurves}}]\label{lemma_g-inv}
Let $G$ be a finite subgroup of the group of automorphisms of a projective variety $X$ over an algebraically closed field of characteristic $p$. Let $\ell$ be a prime number which is coprime both to $p$ and $|G|$ and let $H^i(X,\mathbb{Q}_\ell)^G$ denote the subspace of $G$-invariants in $H^i(X,\mathbb{Q}_\ell)$. Then,
\begin{align*}
H^i(X/G,\mathbb{Q}_\ell)\cong H^i(X,\mathbb{Q}_\ell)^G.
\end{align*}
\end{lemma}
\newpage
Using this lemma, we will prove the following result:
\begin{proposition} \label{module_prop}
Let $A$ be an abelian surface over an algebraically closed field of characteristic $p$ and let $\ell$ and $G$ be as in Lemma \ref{lemma_g-inv}. Assume that $A/G$ is a generalised Kummer surface, so its singular locus only contains rational double points and let $\pi: X\rightarrow A/G$ denote a minimal resolution. Furthermore, let $E_{i,P}$ denote the irreducible exceptional divisors in $X$ arising from the blow-up of each point $P$ in the singular locus of $A/G$. Then, we have the following isomorphism of $\mathbb{Q}_\ell$-vector spaces:
\begin{align*}
H^2(X,\mathbb{Q}_\ell)\cong H^2(A,\mathbb{Q}_\ell)^G\oplus\bigoplus_{P\in\mathrm{Sing}(A/G)\;}\bigoplus_{E_{i,P}\subseteq\pi^{-1}(P)}\mathbb{Q}_\ell(-1)
\end{align*}
\end{proposition}
\begin{proof}
We consider the Leray spectral sequence for the morphism $\pi$ and the constant sheaf $\mathbb{Q}_\ell$:
\begin{align*}
E_2^{p,q} = H^p(A/G, R^q\pi_*\mathbb{Q}_\ell) \Rightarrow H^{p+q}(X, \mathbb{Q}_\ell)
\end{align*}
As $\pi$ is proper, for every geometric point $P\in A/G$ and $q\geq0$, a corollary of the proper base change theorem \cite[Corollaire 5.2]{Deligne1973TheorieSchemas} states that
\begin{align*}
(R^q\pi_*\mathbb{Q}_\ell)_P\cong H^q(X_P,\mathbb{Q}_\ell),
\end{align*}
where $X_P=X\times_{A/G}\{P\}$ is the geometric fibre over $P$. When $P$ is a smooth point of $A/G$, $X_P$ is a single point and, therefore,
\begin{align*}
H^q(X_{P},\mathbb{Q}_\ell)=\begin{cases}
\mathbb{Q}_\ell&\text{if }q=0,\\
0 & \text{otherwise.}\end{cases}\\
\end{align*}
When $P$ is a singular point of $A/G$, as $P$ is a rational double point by hypothesis, $X_P$ is a tree of smooth rational $(-2)$-curves $E_{i,P}$ whose dual graph $\Gamma$ is a simply laced Dynkin diagram of type ADE. From the fact that $\Gamma$ is connected, we deduce that $H^0(X_{P},\mathbb{Q}_\ell)=\mathbb{Q}_\ell$ and from the fact that each curve is isomorphic to $\mathbb{P}^1$ and that $\Gamma$ is a tree, we can use the Mayer–Vietoris long exact sequence to deduce that $H^1(X_{P},\mathbb{Q}_\ell)=0$. We can also check using Mayer-Vietoris that
\begin{align*}
H^2(X_P,\mathbb{Q}_\ell)=\bigoplus_{E_{i,P}\subseteq\pi^{-1}(P)} H^2(E_{i,P},\mathbb{Q}_\ell)=\bigoplus_{E_{i,P}\subseteq\pi^{-1}(P)}\mathbb{Q}_\ell(-1).
\end{align*}
Here, the Tate twist comes from the standard identification
$H^2(\mathbb P^1,\mathbb Q_\ell)\cong \mathbb Q_\ell(-1)$; as we are assuming that we are working over an algebraically closed field, this is simply a one-dimensional $\mathbb Q_\ell$-vector space. Finally, $H^q(X_P,\mathbb{Q}_\ell)=0$ for $q>2$.\\

Combining the information of the fibres, we deduce that $R^q\pi_*\mathbb{Q}_\ell=0$ unless $q=0$, in which case $R^0\pi_*\mathbb{Q}_\ell=\mathbb{Q}_\ell$, or $q=2$. In that case, we have seen that $(R^2\pi_*\mathbb{Q}_\ell)_P$ is zero for all smooth points $P\in A/G$, and therefore, $R^2\pi_*\mathbb{Q}_\ell$ is 
a skyscraper sheaf supported exactly on the singular locus of $A/G$. Its space of global sections is the direct sum of the stalks at each singular point:
\begin{align*}
H^0(A/G, R^2\pi_*\mathbb{Q}_\ell) = \bigoplus_{P \in \mathrm{Sing}(A/G)} H^2(X_P, \mathbb{Q}_\ell)=\bigoplus_{P\in\mathrm{Sing}(A/G)\;}\bigoplus_{E_{i,P}\in\pi^{-1}(P)}\mathbb{Q}_\ell(-1).
\end{align*}
As skyscraper sheaves on finitely many points have no higher cohomology, $H^p(A/G, R^2\pi_*\mathbb{Q}_\ell)=0$ for $p>0$.\\

From the fact that $A/G$ is a surface, we deduce that $H^p(A/G,\mathbb{Q}_\ell)=0$ whenever $p>4$. Therefore, the $E_2$-page of the spectral sequence is 
\[
\begin{array}{c|cccccc}
\vdots&\vdots&\vdots&\vdots&\vdots&\vdots&\iddots\\[5pt]
q=2 & H^0(A/G, R^2\pi_*\mathbb{Q}_\ell)& 0 & 0 & 0 & 0&\cdots \\[5pt]
q=1 & 0 & 0 & 0 & 0 & 0& \cdots \\[5pt]
q=0 & H^0(A/G,\mathbb{Q}_\ell) & H^1(A/G,\mathbb{Q}_\ell) & H^2(A/G,\mathbb{Q}_\ell) & H^3(A/G,\mathbb{Q}_\ell) & H^4(A/G,\mathbb{Q}_\ell)& \cdots \\[5pt]
\hline
 & p=0 & p=1 & p=2 & p=3 & p=4& \cdots
\end{array}
\]

From this description, it is easy to see that the differential maps $d_2^{p,q}:E_2^{p,q}\rightarrow E_2^{p+2,q-1}$ are all zero, so $E_3^{p,q}=E_2^{p,q}$ for all $ p,q\in\mathbb{Z}$. One can also check that $d_3^{p,q}=0$ whenever $(p,q)\neq(0,2)$, as in that case, we have that $d_3^{0,2}:H^0(A/G, R^2\pi_*\mathbb{Q}_\ell)\rightarrow H^3(A/G,\mathbb{Q}_\ell)$.\\

Therefore, we have that $E_4^{p,q}=E_2^{p,q}$ except when $(p,q)=(0,2)$ and $(p,q)=(3,0)$, as in those cases we have that $E_4^{0,2}=\ker(d_3^{0,2})$ and $E_4^{3,0}=\coker(d_3^{0,2})$. Finally, we see that for all $r\geq4$ and $p,q\in\Z$, $d_r^{p,q}=0$ and therefore, 
\begin{align*}
E_\infty^{p,q}=\begin{cases} 
\ker(d_3^{0,2}) &\text{if }(p,q)=(0,2),\\
\coker(d_3^{0,2}) &\text{if }(p,q)=(3,0),\\
E_2^{p,q} &\text{otherwise}.
\end{cases}\\
\end{align*}

As the spectral sequence converges to $H^{p+q}(X, \mathbb{Q}_\ell)$, for each $n$ there is a decreasing filtration
\begin{align*}
H^n(X,\mathbb{Q}_\ell)=F^0H^n(X,\mathbb{Q}_\ell)\supseteq F^1H^n(X,\mathbb{Q}_\ell)\supseteq\dots\supseteq F^{n+1}H^n(X,\mathbb{Q}_\ell)=0
\end{align*}
such that $E_{\infty}^{p,q}\cong F^pH^{p+q}(X,\mathbb{Q}_\ell)/F^{p+1}H^{p+q}(X,\mathbb{Q}_\ell)$ \cite[Section 20.4]{Fomenko2016HomotopicalTopology}.\\

Applying this result to $n=1$, as $E_\infty^{0,1}=0$, we deduce that $H^1(X,\mathbb{Q}_\ell)\cong E_\infty^{1,0}\cong H^1(A/G,\mathbb{Q}_\ell)$. As $X$ is a K3 surface, its first Betti number is zero, and therefore, applying Lemma \ref{lemma_g-inv}, we deduce that
\begin{align*}
H^1(A,\mathbb{Q}_\ell)^G=H^1(A/G,\mathbb{Q}_\ell)=0.\\
\end{align*}
Since $G$ acts on $A$ by automorphisms, all elements $g\in G$ satisfy that $\deg(g)=1$ and thus preserve the trace map $\mathrm{Tr}\colon H^4(A,\mathbb{Q}_\ell(2))\rightarrow\mathbb{Q}_\ell$, as $\mathrm{Tr}(g^*\gamma)=\deg(g)\mathrm{Tr}(\gamma)$. Therefore, the Poincaré pairing
\begin{align*}
H^1(A,\mathbb{Q}_\ell)\times H^3(A,\mathbb{Q}_\ell)\stackrel{\cup}{\longrightarrow}H^4(A,\mathbb{Q}_\ell)\stackrel{\mathrm{Tr}}{\longrightarrow}\mathbb{Q}_\ell(-2)
\end{align*}
is $G$-invariant and we have a $G$-equivariant isomorphism:
\begin{align*}
H^3(A,\mathbb Q_\ell)\cong H^1(A,\mathbb Q_\ell)^\vee(-2).\\  
\end{align*}

Furthemore, since $G$ is finite and $\mathbb Q_\ell$ has characteristic zero,
$\mathbb Q_\ell[G]$ is semisimple by Maschke's theorem. Hence, for any
finite-dimensional $\mathbb Q_\ell[G]$-representation $V$, $(V^\vee)^G\cong(V^G)^\vee$ and therefore
\begin{align*}
H^3(A,\mathbb Q_\ell)^G\cong
(\left(H^1(A,\mathbb Q_\ell)\right)^\vee(-2))^G
\cong
\left(H^1(A,\mathbb Q_\ell)^G\right)^\vee(-2)=0.   \\
\end{align*}
Applying Lemma \ref{lemma_g-inv} again, we deduce that $H^3(A/G,\mathbb Q_\ell)=0$. From this, we deduce that the map $d_3^{0,2}$ must be zero, and therefore, $E^{0,2}_\infty=H^0(A/G, R^2\pi_*\mathbb{Q}_\ell)$ and $E^{3,0}_\infty=0$.\\

When $n=2$, as we have seen that $E_\infty^{1,1}=0$, we deduce that $F^1 H^2(X,\mathbb{Q}_\ell)= F^2 H^2(X,\mathbb{Q}_\ell)$. Therefore,
\begin{align*}
H^2(A/G,\mathbb{Q}_\ell)&\cong F^2 H^2(X,\mathbb{Q}_\ell),&H^0(A/G, R^2\pi_*\mathbb{Q}_\ell)&\cong H^{2}(X,\mathbb{Q}_\ell)/F^{2}H^{2}(X,\mathbb{Q}_\ell).\\
\end{align*}
Hence there is a short exact sequence
\begin{align*}
0 \to H^2(A/G, \mathbb{Q}_\ell) \xrightarrow{\pi^*} H^2(X, \mathbb{Q}_\ell) \to H^0(A/G, R^2\pi_*\mathbb{Q}_\ell) \to 0.\\
\end{align*}
Since all terms are $\mathbb{Q}_\ell$-vector spaces, this sequence splits and therefore,
\begin{align*}
H^2(X, \mathbb{Q}_\ell)&\cong H^2(A/G, \mathbb{Q}_\ell)\oplus\,H^0(A/G, R^2\pi_*\mathbb{Q}_\ell)\\[5pt]
&\cong H^2(A/G, \mathbb{Q}_\ell) \oplus \bigoplus_{P \in \mathrm{Sing}(A/G)} \bigoplus_{E_{i,P} \subseteq \pi^{-1}(P)} \mathbb{Q}_\ell(-1).
\end{align*}

Finally, the result follows from the fact that $H^2(A/G, \mathbb{Q}_\ell)\cong H^2(A, \mathbb{Q}_\ell)^G$ by Lemma \ref{lemma_g-inv}.
\end{proof}
\begin{remark}
As $b_4(X)=b_4(A/G)=1$, the cup product endows the second cohomology groups of $X$ and $A/G$ with a bilinear product. The fact that $H^0(A/G, R^2\pi_*\mathbb{Q}_\ell)$ embeds into $H^2(X, \mathbb{Q}_\ell)$ allows us to extend the result in Proposition \ref{module_prop} to an isomorphism of $\mathbb{Q}_\ell$-algebras.  When $k$ is not algebraically closed, the isomorphism is 
$\mathrm{Gal}(\bar{k}/k)$-equivariant and therefore preserves the Galois-module structure on the cohomology groups. This has been proved for Kummer surfaces by Overkamp \cite{Overkamp2021DegenerationSurfaces} and Lazda–Skorobogatov \cite{Lazda2023ReductionCase}, and very recently for generalised Kummer surfaces by Zhao \cite[Proposition 2.10]{Zhao2026TheCase}.
\end{remark}\newpage
\begin{corollary} \label{three-five_cor}
Let $A$ be a supersingular abelian surface in characteristic $p\in\{3,5\}$ and let $C_p$ act on $A$. Then, $A/C_p$ is not a generalised Kummer surface.
\end{corollary}
\begin{proof}
For $A/C_p$ to be a generalised Kummer surface, $C_p$ has to act rigidly on $A$ and the singularities of $A/C_p$ can only possibly be rational double points. By Proposition \ref{fixorig_prop}, as $C_p$ only fixes the origin, $A/C_p$ can only possibly have one singular point, and from Lemma \ref{prop_fund}, as the local fundamental group of the singular point must contain $C_p$, we deduce that the only possibilities are that the singular point is of type $E_6^1$, $E_7^1$ or $E_8^2$ if $p=3$ or $E_8^1$ if $p=5$. In particular, the number of exceptional curves $n$ in the desingularisation of $A/C_p$ would satisfy that $n\in\{6,7,8\}$ if $p=3$ and $n=8$ if $p=5$.\\

For any K3 surface $X$, $\dim(H^2(X,\mathbb{Q}_\ell))=22$ and for every abelian surface $A$, $\dim(H^2(A,\mathbb{Q}_\ell))=6$ \cite{Liedtke2013AlgebraicCharacteristic}. Therefore, we deduce that $A/C_p$ cannot be a generalised Kummer surface, as by Proposition \ref{module_prop}, we would deduce that the number of exceptional divisors $n$ in the resolution of the singular point should be
\begin{align*}
&n=\dim(H^2(X,\mathbb{Q}_\ell))-\dim(H^2(A,\mathbb{Q}_\ell)^G)\geq \dim(H^2(X,\mathbb{Q}_\ell))-\dim(H^2(A,\mathbb{Q}_\ell))=16.& &\qedhere
\end{align*}
\end{proof}
Combining Proposition \ref{quot_prop} with Corollaries \ref{two_cor} and \ref{three-five_cor}, we deduce the following:
\begin{theorem}
If $A$ is a supersingular abelian surface in characteristic $p$ and $G$ a group acting on $A$ such that $p\mid |G|$, the desingularisation of $A/G$ is never a K3 surface.\\
\end{theorem}

\section{Examples of quotients of products of elliptic curves} \label{sec_examples} 
In Section \ref{computingsings_section}, we compiled a list of possible groups $G$ and conditions on the abelian surface $A$ under which $A/G$ could be a generalised Kummer surface in positive characteristic. However, it is not immediately clear why such surfaces should actually exist.\\

In this section, we will show that they do exist by constructing explicit examples in the special case where $A$ is the product of two elliptic curves. The reason why it is convenient to work in this setup is that, particularly when $A = E \times E$, the endomorphism ring $\End(A)$ admits a simple description: it is isomorphic to the matrix ring $M_2(\End(E))$. As a result, group actions on $A$ can be understood as matrix groups with entries in the ring of integers of a number field.\\

Recall from Table \ref{table1} that there are generalised Kummer surfaces that occur only in positive characteristic, and only when $A$ is a supersingular abelian surface. These are the cases where $G$ is $C_5$, $C_{10}$, $C_{12}$, $Q_{16}$, $Q_{24}$, $\ESL_2(\FF_3)$, and $\SL_2(\FF_3)$. The existence of generalised Kummer surfaces for the cyclic groups in this list was proven by Katsura \cite[Examples 1--11]{Katsura1987GeneralizedP}, and for the remaining groups, by Rybakov \cite[Lemmas 6.6 and 6.7]{Rybakov2024GeneralizedFields}. The construction of these examples is somewhat involved, and for that reason, we will not discuss them in this paper.\\

Instead, we will focus on the cases where $A$ is not supersingular. For cyclic quotients, we will describe a geometric method for constructing a rigid and symplectic action on a product of two elliptic curves, and explain how to explicitly compute the quotient surface under this action.\\

We will also provide examples of the actions of $Q_8$, $Q_{12}$, and $\SL_2(\FF_3)$ on products of elliptic curves. Furthermore, by analysing the fixed points of these actions, we will describe the singularities of the corresponding quotients. However, we will not give explicit defining equations for these quotient surfaces. All computations described in this section, together with their implementations in Magma, are available in \href{https://github.com/AlvaroGohe/Generalised-Kummer-surfaces}{this repository}.\\


Let $E$ be an elliptic curve given in Weierstrass form by $\mathbb{V}(f) \subset \mathbb{P}^2$, where
\begin{align*}
f = y^2z + a_1xyz + a_3yz^2 - (x^3 + a_2 x^2z + a_4 xz^2 + a_6z^3).
\end{align*}

We will now explain how to construct the quotient of the product of two elliptic curves by actions of the groups $C_2$, $C_3$, $C_4$, or $C_6$.\\
\newpage
\subsection{The action of order two} \label{order2_subsection}

The Kummer surface associated with the product of two elliptic curves has been extensively studied. For the sake of completeness, we describe here how to compute a model in a similar spirit to the description given by Shioda \cite{Shioda1974Kummer2}.\\

Let $E_1$ and $E_2$ be elliptic curves given in Weierstrass form by the equations
\begin{align*}
E_1\colon\quad & y^2z + a_1xyz + a_3yz^2 = x^3 + a_2 x^2z + a_4 xz^2 + a_6z^3, \\
E_2\colon\quad & y^2z + b_1xyz + b_3yz^2 = x^3 + b_2 x^2z + b_4 xz^2 + b_6z^3.\\
\end{align*}

The involution that sends any element of $E_1 \times E_2$ to its inverse is given by the map\\[5pt]
\begin{adjustbox}{width=\textwidth}
\begin{minipage}{1.1\textwidth}
\begin{align*}
\iota \colon\quad E_1 \times E_2 &\longrightarrow E_1 \times E_2 \\
([x_1:y_1:z_1], [x_2:y_2:z_2]) &\longmapsto ([x_1:-y_1 - a_1x_1 - a_3z_1 : z_1], [x_2:-y_2 - b_1x_2 - b_3z_2 : z_2]).\\[5pt]
\end{align*}
\end{minipage}
\end{adjustbox}

We now construct the Kummer surface $\Kum(E_1 \times E_2)$, which is the quotient of $E_1 \times E_2$ by the action of $\iota$. We define a grading on $\mathbb{P}^2 \times \mathbb{P}^2$ by setting the multidegree of $\{x_1, y_1, z_1\}$ to be $\binom{1}{0}$ and the multidegree of $\{x_2, y_2, z_2\}$ to be $\binom{0}{1}$. It is easy to see that the functions $\{x_1, z_1, x_2, z_2\}$ are invariant by the action of $\iota$. Moreover, the function
\begin{align} \label{kum2coord}
w = z_1z_2 \big( 2y_1y_2 + (a_1x_1 + a_3z_1)y_2 + (b_1x_2 + b_3z_2)y_1 \big)
\end{align}
is also invariant under $\iota$ and has multidegree $\binom{2}{2}$. One can check that $x_1, z_1, x_2, z_2$ and $w$ generate all invariant functions in the function field of $E_1 \times E_2$. These invariants satisfy a relation of the form
\begin{align} \label{rel2}
w^2 + w z_1z_2(a_1 x_1 + a_3 z_1)(b_1 x_2 + b_3 z_2) = z_1z_2 f_6(x_1, x_2, z_1, z_2),
\end{align}
for some degree six polynomial $f_6 \in k[x_1, x_2, z_1, z_2]$.\\

Geometrically, there is another way to arrive at this model, which we will now describe, as it will connect nicely with the method used to compute the quotient by the action of $C_4$. Consider the map
\begin{align*}
\psi \colon\quad  E_1 \times E_2 &\longrightarrow \mathbb{P}^1 \times \mathbb{P}^1 \\
([x_1:y_1:z_1], [x_2:y_2:z_2]) &\longmapsto ([x_1 : z_1], [x_2 : z_2]),
\end{align*}
which assigns to a pair of points $(P_1, P_2)$ the vertical lines $\ell_{P_1}$ and $\ell_{P_2}$ passing through them. Note that this map is not defined at the curves $\{O\} \times E_2$ and $E_1 \times \{O\}$, but one can check that it extends to a morphism. Observe that $(P_1, P_2)$ and $(-P_1, -P_2)$ are mapped to the same point in $\mathbb{P}^1 \times \mathbb{P}^1$ under $\psi$, and therefore the map descends to a morphism
\begin{align*}
\psi_K \colon\quad \Kum(E_1 \times E_2) \rightarrow \mathbb{P}^1 \times \mathbb{P}^1.\\[-10pt]
\end{align*}

This is generically a $2$-to-$1$ cover: the fibre over a generic point $(\ell_{P_1}, \ell_{P_2})$ consists of the four points $(\pm P_1, \pm P_2)$, which yield two distinct points in the Kummer surface.\\

Assuming that the characteristic of the base field is not two, $\Kum(E_1 \times E_2)$ can be described as a double cover branched along the ramification locus of $\psi_K$. The ramification occurs precisely when $P_1 = -P_1$ or $P_2 = -P_2$. Therefore, $\Kum(E_1 \times E_2)$ can be defined by the equation
\[
v^2 = z_1z_2(x_1 - \alpha_1z_1)(x_1 - \alpha_2z_1)(x_1 - \alpha_3z_1)(x_2 - \beta_1z_2)(x_2 - \beta_2z_2)(x_2 - \beta_3z_2),
\]
where $\alpha_i$ and $\beta_j$ are the $x$-coordinates of the $2$-torsion points of $E_1$ and $E_2$, respectively.
 This means that\\
\begin{adjustbox}{width=\textwidth}
\begin{minipage}{1.1\textwidth}
\begin{align*}
v^2=z_1z_2(x_1^3+a_2 x_1^2z_1+a_4 x_1z_1^2+a_6z_1^3+\tfrac{1}{4}(a_1x_1+a_3z_1)^2)(x_2^3+b_2 x_2^2z_2+b_4 x_2z_2^2+b_6z_2^3+\tfrac{1}{4}(b_1 x_2 + b_3 z_2)^2)\\[5pt]
\end{align*}
\end{minipage}
\end{adjustbox}
The singular locus of this surface is one-dimensional for fields of characteristic two; however, by considering the change of variables
\begin{align*}
w=2v-\tfrac{1}{2}(a_1 x_1 + a_3 z_1) (b_1 x_2 + b_3 z_2).
\end{align*}
Clearing the powers of two of the denominator, we recover the Equation \ref{rel2}, which only has isolated singularities.\\

Therefore, $\Kum(E_1\times E_2)$ admits a model as a hypersurface $X$ inside the scroll $\mathbb{P}(M)$ where
\begin{align*}
M=\begin{pmatrix}
1 &1 &0 &0 &2\\
0 &0& 1 &1 &2\\
\end{pmatrix}.
\end{align*}

While this model works well if we want to work in affine patches of it, it is not very convenient for other tasks such as computing the singular points of the variety. For that purpose, we can re-embed our variety inside $\mathbb{P}(1,1,1,1,2)$ by considering the image through the following map:
\begin{align*}
X&\longrightarrow\mathbb{P}(1,1,1,1,2)\\
[x_1:z_1:x_2:z_2:w]&\longmapsto[x_1x_2:x_1z_2:x_2z_1:z_1z_2:w]\\[-5pt]
\end{align*}
Then, $\Ktwo$ is given by the intersection of a polynomial of weighted degree two, and the image of the Equation \ref{rel2} inside $\mathbb{P}(1,1,1,1,2)$, which has weighted degree four.\\

If the characteristic of the base field is not two, this surface has sixteen singularities of type $A_1$ corresponding to the fixed points of $E_1\times E_2$ by the action, which is its $2$-torsion subgroup.\\

If the characteristic of the base field is two, then the number and type of singular points depend on the $p$-rank of $E_1\times E_2$. 
\begin{itemize}
    \item If the $p$-rank is $2$, it has four rational singularities of type $D^1_4$.
    \item If the $p$-rank is $1$, it has two rational singularities of type $D^2_8$.
    \item If the $p$-rank is $0$, it has one elliptic singularity of type $\Circled{19}_{\,0}$ in the sense of Wagreich \cite{Wagreich1970EllipticSurfaces} or $A_{\ast,o}+A_{\ast,o}+A_{\ast,o}+A_{\ast,o}+A_{\ast,o}$ in the sense of Laufer \cite{Laufer1977OnSingularities}. In this case, the Kummer surface associated to this surface is not a K3 surface, as a consequence of Proposition \ref{ellsing_prop}.
 
\end{itemize}
\begin{figure}[H]
\begin{center}
\begin{tikzpicture}
	\begin{pgfonlayer}{nodelayer}
		\node [style=black with white font] (1) at (0, 0) {$-3$};
		\node [style=white with black border] (3) at (0.951057, 0.309017) {};
		\node [style=white with black border] (7) at (0, 1) {};
		\node [style=white with black border] (8) at (-0.951057, 0.309017) {};
		\node [style=white with black border] (9) at (0.587785, -0.809017) {};
		\node [style=white with black border] (10) at (-0.587785, -0.809017) {};
	\end{pgfonlayer}
	\begin{pgfonlayer}{edgelayer}
		\draw (1) to (7);
		\draw (1) to (8);
		\draw (1) to (10);
		\draw (1) to (9);
		\draw (1) to (3);
	\end{pgfonlayer}
\end{tikzpicture}
\vspace{5pt}
\caption{Resolution graph of the $A_{\ast,o}+A_{\ast,o}+A_{\ast,o}+A_{\ast,o}+A_{\ast,o}$ singularity.}
\end{center}
\end{figure}
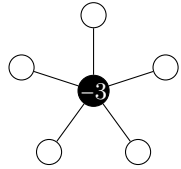

\subsection{The action of order three}  \label{order3_subsection}

This construction has been described by many authors, most notably by van Luijk and by Fearnley, Kisilevsky and Kuwata in the context of studying cubic points of cubic curves \cite{vanLuijk2007CubicSurfaces, Fearnley2012VanishingCurves}. The explicit construction of $(E \times E)/C_3$ was recently described by Kondō and Mukai \cite[Section 7.2]{Kondo2024TheTwo} for elliptic curves expressed as cubics of the form
\begin{align*}
E:\quad x^3 + y^3 + z^3 - \lambda xyz = 0.\\[-5pt]
\end{align*}
We follow similar methods to construct $(E \times E)/C_3$, but for $E$ given in Weierstrass form. We define an automorphism of $E \times E$ of order three  by
\begin{align*}
\tau_3:\quad E\times E&\longrightarrow E\times E\\
(P,Q)&\longmapsto (-P-Q,P)\\
\end{align*}
Geometrically, this corresponds to the following operation: for every two points $P, Q \in E$, consider the line through them, which intersects $E$ in a third point $-P - Q$. Then, the action $\tau_3$ permutes cyclically the pairs of points in the line, as we can see in the diagram:\\

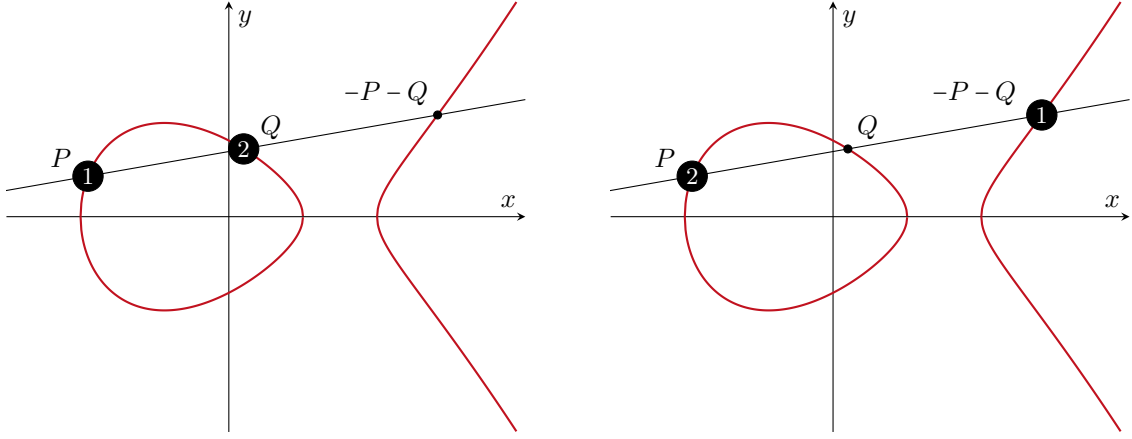
\begin{figure}[H]
\centering
\begin{tikzpicture}
    \begin{axis}[
        axis lines=center,
        xlabel={$x$},
        ylabel={$y$},
        xtick=\empty,
        ytick=\empty,
        xmin=-1.5, xmax=2,
        ymin=-2, ymax=2,
        samples=200,
        domain=-1.5:2,
        restrict y to domain=-2:2,
    ]
    
        \addplot [
            darkred,
            thick,
            samples=300,
            domain=-1:1/2,
        ]
        ({\x}, {sqrt((\x - 1/2)*(\x - 1)*(\x + 1))});

        \addplot [
            darkred,
            thick,
            samples=300,
            domain=-1:1/2,
        ]
        ({\x}, {-sqrt((\x - 1/2)*(\x - 1)*(\x + 1))});
        
          \addplot [
            darkred,
            thick,
            samples=200,
            domain=1:2,
        ]
        ({\x}, {sqrt((\x - 1/2)*(\x - 1)*(\x + 1))});

        \addplot [
            darkred,
            thick,
            samples=200,
            domain=1:2,
        ]
        ({\x}, {-sqrt((\x - 1/2)*(\x - 1)*(\x + 1))});

        \addplot[domain=-2:2] {0.375999 + 0.241225*(0.95 + x)};
        
        \node[above left] at (axis cs:-1, 0.375999) {$P$};
        \node[fatpoint] (P) at (axis cs:-0.95, 0.375999) {};
        \node[pointlabel] at (P) {1};
        \node[above right] at (axis cs:0.15, 0.629285) {$Q$};
        \node[fatpoint] (Q) at (axis cs:0.1, 0.629285) {};
        \node[pointlabel] at (Q) {2};
        
        \addplot[only marks, mark=*, mark size=1.5pt] coordinates {(1.40819, 0.944854)};
        \node[above left] at (axis cs:1.40819, 0.944854) {$-P-Q$};

    \end{axis}
\end{tikzpicture} \hspace{25pt}
\begin{tikzpicture}
    \begin{axis}[
        axis lines=center,
        xlabel={$x$},
        ylabel={$y$},
        xtick=\empty,
        ytick=\empty,
        xmin=-1.5, xmax=2,
        ymin=-2, ymax=2,
        samples=200,
        domain=-1.5:2,
        restrict y to domain=-2:2,
    ]
    
        \addplot [
            darkred,
            thick,
            samples=300,
            domain=-1:1/2,
        ]
        ({\x}, {sqrt((\x - 1/2)*(\x - 1)*(\x + 1))});

        \addplot [
            darkred,
            thick,
            samples=300,
            domain=-1:1/2,
        ]
        ({\x}, {-sqrt((\x - 1/2)*(\x - 1)*(\x + 1))});
        
          \addplot [
            darkred,
            thick,
            samples=200,
            domain=1:2,
        ]
        ({\x}, {sqrt((\x - 1/2)*(\x - 1)*(\x + 1))});

        \addplot [
            darkred,
            thick,
            samples=200,
            domain=1:2,
        ]
        ({\x}, {-sqrt((\x - 1/2)*(\x - 1)*(\x + 1))});

        \addplot[domain=-2:2] {0.375999 + 0.241225*(0.95 + x)};
        
        \node[above left] at (axis cs:-1, 0.375999) {$P$};
        \node[fatpoint] (P) at (axis cs:-0.95, 0.375999) {};
        \node[pointlabel] at (P) {2};
        \node[above left] at (axis cs:1.3, 0.944854) {$-P-Q$};
        \node[fatpoint] (Q) at (axis cs:1.40819, 0.944854) {};
        \node[pointlabel] at (Q) {1};
        
        \addplot[only marks, mark=*, mark size=1.5pt] coordinates {(0.1, 0.629285)};
        \node[above right] at (axis cs:0.1, 0.629285) {$Q$};
        
    \end{axis}
\end{tikzpicture} 
\caption{Action of $\tau_3$ on $E\times E$.}
\end{figure}
Let $\Kthree$ be the quotient of $E \times E$ by this action. We can find a model for $\Kthree$ as follows. Each pair $(P, Q)$ determines a line $\ell_{P,Q}$ in $\mathbb{P}^2$, and this line corresponds to a point in the dual projective plane $\mathbb{P}^{2\vee}$. Therefore, we can define a map
\begin{align*}
\psi\colon\quad E\times E&\longrightarrow \mathbb{P}^{2,\vee}\\
(P,Q)&\longmapsto \ell_{P,Q}\\
\end{align*}

As for any two points $P,Q$, any pair of points in the set $\{P,Q,-P-Q\}$ is sent to the same line, the map $\psi$ factors through the quotient. Hence, we obtain an induced map $\psi_K: \Kthree \rightarrow \mathbb{P}^2$. Generically, the preimage of a point under $\psi$ consists of $3!=6$ ordered pairs $(P, Q)$, but three of these are identified in $\Kthree$, so $\psi_K$ is a double cover of $\mathbb{P}^2$.\\

To describe $\Kthree$ explicitly, it suffices to compute the ramification locus of $\psi_K$. A point $(P, Q)$ lies in the ramification locus if and only if two of the pairs in the orbit $(P, Q)$, $(-P - Q, P)$ and $(Q, -P - Q)$ are equal in $\Kthree$. This happens precisely when the line $\ell_{P,Q}$ is tangent to $E$ at one of the three points. Therefore, the ramification locus is the dual curve $E^\vee \subset \mathbb{P}^{2\vee}$, which is the image of the map
\begin{align*}
\phi\colon\quad E&\longrightarrow \mathbb{P}^{2,\vee}\\
P&\longmapsto\Big[\frac{\partial f}{\partial x}(P):\frac{\partial f}{\partial y}(P):\frac{\partial f}{\partial z}(P)\Big]\\[-5pt]
\end{align*}
This curve $E^\vee$ is a sextic, defined by the vanishing of a homogeneous degree six polynomial $h_6 \in k[u,v,w]$. Therefore, $\Kthree$ can be described as a double cover of $\mathbb{P}^2$ branched over $\mathbb{V}(h_6)$, and thus given by an equation of the form
\begin{equation} \label{equation1}
t^2 = h_6(u, v, w)
\end{equation}
in weighted projective space $\mathbb{P}(1,1,1,3)$, assuming the characteristic is not two.\\

In characteristic two, this model of a curve would give rise to an inseparable cover. To construct a valid model in this case, we start in characteristic zero, find a model there that reduces well at two, and then specialise it to characteristic two. When we reduce modulo $2$ the coefficients of the model given by Equation \ref{equation1}, we obtain that $t^2 = g_3^2$ for some polynomial $g_3$ of degree three. Then, by lifting $g_3$ to characteristic zero and setting
\[
f_6 = \frac{1}{4}(h_6 - g_3^2),
\]
we obtain the following model of $\Kthree$ in characteristic zero which reduces well at two
\[
t^2 + g_3(u,v,w)t + f_6(u,v,w) = 0.\\
\]
\newpage

In characteristic different from three, $\Kthree$ has nine $A_2$ singularities, arising from the fixed points of the action of $C_3$. These correspond to the inflection points of $E$, which are the $3$-torsion points $E[3]$.\\

In characteristic three, the fixed locus of the action depends on whether $E$ is ordinary or supersingular. If $E$ is ordinary, then $E[3] \cong \mathbb{Z}/3\mathbb{Z}$, and $\Kthree$ has three singularities. The resolution of these singularities reveals that the intersection matrix of the exceptional divisors is of type $E_6$, and the Tjurina number is seven, showing that the singularities are of type $E_6^1$.\\

If $E$ is supersingular, then we can check that $(E\times E)/C_3$ has a single elliptic singularity, of type $D_{4,***}$ in the notation of Laufer \cite{Laufer1977OnSingularities}.
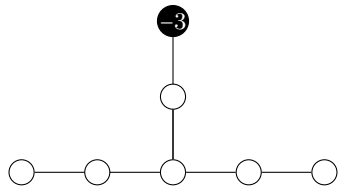
\begin{figure}[H]
\begin{center}
\begin{tikzpicture}
	\begin{pgfonlayer}{nodelayer}
		\node [style=black with white font] (1) at (2, 2) {$-3$};
		\node [style=white with black border] (3) at (2, 1) {};
		\node [style=white with black border] (4) at (2, 0) {};
		\node [style=white with black border] (5) at (3, 0) {};
		\node [style=white with black border] (7) at (1, 0) {};
		\node [style=white with black border] (8) at (0, 0) {};
		\node [style=white with black border] (9) at (4, 0) {};
	\end{pgfonlayer}
	\begin{pgfonlayer}{edgelayer}
		\draw [style=new edge style 0] (1) to (4);
		\draw [style=new edge style 0] (4) to (3);
		\draw [style=new edge style 0] (4) to (5);
		\draw (8) to (4);
		\draw (5) to (9);
	\end{pgfonlayer}
\end{tikzpicture}
\end{center}
\caption{Resolution graph of the $D_{4,***}$ singularity.}
\end{figure}

In the cases where the singularities are rational double points, applying the adjunction formula, one can check that the canonical divisor of $\Kthree$ is trivial and therefore, these are K3 surfaces. If the base field has characteristic three and $E$ is supersingular, $(E\times E)/C_3$ is a rational surface by Proposition \ref{ellsing_prop}.

\subsection{The action of order four} \label{order4_subsection}
 Whenever we have that $A=E\times E$, we can define an action of order $4$ on $A$ by
\begin{align*}
\tau_4:\quad E\times E&\longrightarrow E\times E\\
(P,Q)&\longmapsto (-Q,P)\\[-5pt]
\end{align*}

Then, $\tau_4(\tau_4(P,Q))=(-P,-Q)$ and therefore, $\tau_4^2$ is the involution $\iota$. The action can be described through the following diagram:
\begin{figure}[H]
\begin{center}
\begin{tikzpicture}
    \begin{axis}[
        axis lines=center,
        xlabel={$x$},
        ylabel={$y$},
        xtick=\empty,
        ytick=\empty,
        xmin=-1.5, xmax=2,
        ymin=-2, ymax=2,
        samples=200,
        domain=-1.5:2,
        restrict y to domain=-2:2,
    ]
    
        \addplot [
            darkred,
            thick,
            samples=300,
            domain=-1:1/2,
        ]
        ({\x}, {sqrt((\x - 1/2)*(\x - 1)*(\x + 1))});

        \addplot [
            darkred,
            thick,
            samples=300,
            domain=-1:1/2,
        ]
        ({\x}, {-sqrt((\x - 1/2)*(\x - 1)*(\x + 1))});
        
          \addplot [
            darkred,
            thick,
            samples=200,
            domain=1:2,
        ]
        ({\x}, {sqrt((\x - 1/2)*(\x - 1)*(\x + 1))});

        \addplot [
            darkred,
            thick,
            samples=200,
            domain=1:2,
        ]
        ({\x}, {-sqrt((\x - 1/2)*(\x - 1)*(\x + 1))});

        \addplot[] coordinates {(-0.95, 2) (-0.95,-2)};
        \addplot[] coordinates  {(1.40819, 2) (1.40819,-2)};

        \node[above left] at (axis cs:-1, 0.375999) {$P$};
        \node[fatpoint] (P) at (axis cs:-0.95, 0.375999) {};
        \node[pointlabel] at (P) {1};
        \node[above left] at (axis cs:1.35, 0.944854) {$Q$};
        \node[fatpoint] (Q) at (axis cs:1.40819, 0.944854) {};
        \node[pointlabel] at (Q) {2};
        \addplot[only marks, mark=*, mark size=1.5pt] coordinates {(-0.95, -0.375999) (1.40819, -0.944854) };
        \node[below left] at (axis cs:1.40819, -0.944854) {$-Q$};
        \node[below left] at (axis cs:-0.95, -0.375999) {$-P$};

    \end{axis}
\end{tikzpicture} \hspace{25pt}
\begin{tikzpicture}
    \begin{axis}[
        axis lines=center,
        xlabel={$x$},
        ylabel={$y$},
        xtick=\empty,
        ytick=\empty,
        xmin=-1.5, xmax=2,
        ymin=-2, ymax=2,
        samples=200,
        domain=-1.5:2,
        restrict y to domain=-2:2,
    ]
    
        \addplot [
            darkred,
            thick,
            samples=300,
            domain=-1:1/2,
        ]
        ({\x}, {sqrt((\x - 1/2)*(\x - 1)*(\x + 1))});

        \addplot [
            darkred,
            thick,
            samples=300,
            domain=-1:1/2,
        ]
        ({\x}, {-sqrt((\x - 1/2)*(\x - 1)*(\x + 1))});
        
          \addplot [
            darkred,
            thick,
            samples=200,
            domain=1:2,
        ]
        ({\x}, {sqrt((\x - 1/2)*(\x - 1)*(\x + 1))});

        \addplot [
            darkred,
            thick,
            samples=200,
            domain=1:2,
        ]
        ({\x}, {-sqrt((\x - 1/2)*(\x - 1)*(\x + 1))});

        \addplot[] coordinates {(-0.95, 2) (-0.95,-2)};
        \addplot[] coordinates  {(1.40819, 2) (1.40819,-2)};

        \node[below left] at (axis cs:-1, -0.375999) {$-P$};
        \node[fatpoint] (P) at (axis cs:-0.95, -0.375999) {};
        \node[pointlabel] at (P) {2};
        \node[above left] at (axis cs:1.35, 0.944854) {$Q$};
        \node[fatpoint] (Q) at (axis cs:1.40819, 0.944854) {};
        \node[pointlabel] at (Q) {1};
        \addplot[only marks, mark=*, mark size=1.5pt] coordinates {(-0.95, 0.375999) (1.40819, -0.944854) };
        \node[below left] at (axis cs:1.40819, -0.944854) {$-Q$};
        \node[above left] at (axis cs:-0.95, 0.375999) {$P$};
    \end{axis}
\end{tikzpicture}
\end{center}
\caption{Action of $\tau_4$ on $E\times E$.}
\end{figure}
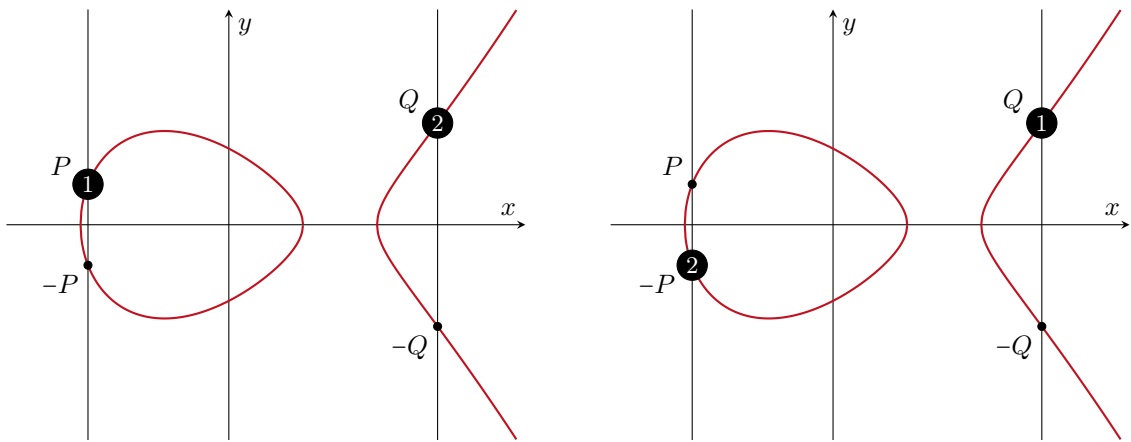
For constructing the quotient of $E \times E$ by the action of order two, we studied the map that assigns to every ordered pair of points the corresponding ordered pair of vertical lines passing through them. For the quotient by the action of order four, we will instead consider a map $\varsigma$ that assigns to each pair of points the unordered pair of vertical lines passing through them.\\

Let $\tau_2:\mathbb{P}^1\times\mathbb{P}^1\rightarrow\mathbb{P}^1\times\mathbb{P}^1$ be the morphism that swaps the two copies of $\mathbb{P}^1$.
The map $\varsigma$ is the composition of the map $\psi: E \times E\rightarrow\mathbb{P}^1\times\mathbb{P}^1$ that we defined in Subsection \ref{order2_subsection} with the quotient map $\pi_2:\mathbb{P}^1\times\mathbb{P}^1\rightarrow(\mathbb{P}^1\times\mathbb{P}^1)/\langle\tau_2\rangle$. In coordinates, $\varsigma$ is given by
\begin{align*}
\varsigma\colon\quad E \times E &\longrightarrow \mathbb{P}^2 \\
([x_1:y_1:z_1],[x_2:y_2:z_2]) &\longmapsto [x_1x_2 : z_1z_2 : x_1z_2 + x_2z_1]\\[-5pt]
\end{align*}

As with the previous case, it is straightforward to check that this map descends to a morphism on the quotient surface $\Kfour$, and that the induced map is a double cover. Assuming that the characteristic of the base field is not $2$, the ramification locus of this map corresponds to the union of the following curves in $\mathbb{P}^2$:
\begin{enumerate}
    \item The locus of unordered pairs of vertical lines that coincide forms a conic in $\mathbb{P}^2$.
    \item The locus of unordered pairs where one of the lines is tangent to a $2$-torsion point of $E$ are four lines in $\mathbb{P}^2$, each line corresponding to a $2$-torsion point.
\end{enumerate}

These five curves form a highly symmetric configuration in the plane: the four lines all intersect one another, and the conic is tangent to each of the four lines, as illustrated in the diagram below:
\begin{figure}[H]
\begin{center}
\begin{tikzpicture}
\begin{axis}
[
    axis equal,
    axis line style={draw=none}, 
    tick style={draw=none},      
    xtick=\empty,                
    ytick=\empty,                
    xlabel={},                   
    ylabel={}                    
]
        \addplot [
            darkred,
            thick,
            samples=201,
            domain=-1:1,
        ]
        ({\x}, {sqrt(1-\x*\x)});

        \addplot [
            darkred,
            thick,
            samples=201,
            domain=-1:1,
        ]
        ({\x}, {-sqrt(1-\x*\x)});
    
       \addplot[domain=-0.5:2.7] {2.6 - 2.4*x};
       \addplot[domain=-2.7:0.5] {2.6 + 2.4*x};
       \addplot[domain=-3.5:3.5] {-1.25 + 0.75*x};
       \addplot[domain=-3.5:3.5] {-1.25 - 0.75*x};
       \addplot[only marks, mark=*, mark size=1.5pt] coordinates {(0, 2.6) (1.22222, -0.333333) (2.33333, -3) (-2.33333, -3) (-1.22222, -0.333333) (0, -1.25) };
       \addplot[only marks, mark=*, mark size=1.5pt, color=darkred] coordinates {(0.923077, 0.384615) (-0.923077, 0.384615) (0.6, -0.8) (-0.6, -0.8)};
\end{axis}
\end{tikzpicture}
\end{center}
\caption{Ramification locus of the map $\Kfour\rightarrow\mathbb{P}^2$.}
\end{figure}
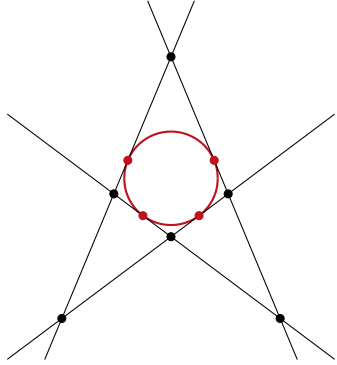
This construction gives rise to a model of $\Kfour$ as a degree six hypersurface inside $\mathbb{P}(1,1,1,3)$. This surface has four singularities of type $A_3$, which are the images in the quotient of the points of the form $(P, P)$ for some $P \in E[2]$, and six singularities of type $A_1$, which arise from points $(P, Q)$ with $P, Q \in E[2]$ and $P \neq Q$. There are twelve such points, but they are identified $2$-to-$1$ by the quotient.\\

In the description of the surface as a double cover, all singularities correspond to singular points of the ramification locus. Specifically, the $A_3$ singularities are the tangency points of the conic with the lines, while the $A_1$ singularities correspond to the pairwise intersections of the lines.

\begin{remark}
Given any configuration of four lines and a conic in $\mathbb{P}^2$ as described above, it is not difficult to show that, up to isomorphism, one can always construct an elliptic curve $E$ such that the surface $X_6 \subset \mathbb{P}(1,1,1,3)$, defined as the double cover ramified along the configuration, is isomorphic to $\Kfour$.
\end{remark}

To obtain a description that also works in characteristic two, we construct a set of functions that are invariant by the action of $\tau_4$. These define the following map:\\
\begin{adjustbox}{width=\textwidth}
\begin{minipage}{1.1\textwidth}
\begin{align*}
E \times E &\longrightarrow \mathbb{P}(1,1,1,3) \\
(P_1, P_2) &\longmapsto [x_1x_2 : z_1z_2 : x_1z_2 + x_2z_1 : (x_1z_2 - x_2z_1)v - x_2z_1^2z_2(a_1x_1 + a_3z_1)(a_1x_2 + a_3z_2)],\\
\end{align*}
\end{minipage}
\end{adjustbox}
where $v$ is the function defined in equation~\eqref{kum2coord}. The image of this map is a sextic hypersurface, giving a model of $\Kfour$ that is valid in any characteristic.\newpage In particular, in characteristic two, if $E$ is an ordinary elliptic curve, $\Kfour$ is a K3 surface with two singularities of type $E_7^3$, corresponding to the images of the points $(P, P)$ for the two distinct $2$-torsion points $P \in E[2]$, and a singularity of type $D_4^1$ arising from the two points $(P, Q)$ with $P \neq Q$.\\

If $E$ is the unique supersingular elliptic curve in characteristic two, $\Kfour$ is a rational surface. This can be deduced either from the fact that $\Kfour$ is a quotient of $\Ktwo$, which is already rational, or from the fact that $\Kfour$ has an elliptic singularity of type $A_{*,o} + A_{*,o} + A_{*,o} + A_{2,**,o}$, in the notation of Laufer \cite{Laufer1977OnSingularities}.
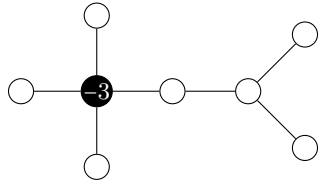
\begin{figure}[H]
\begin{center}
    \begin{tikzpicture}
	\begin{pgfonlayer}{nodelayer}
		\node [style=black with white font] (1) at (0, 0) {$-3$};
		\node [style=white with black border] (3) at (1, 0) {};
		\node [style=white with black border] (4) at (2, 0) {};
		\node [style=white with black border] (5) at (2.75, 0.75) {};
		\node [style=white with black border] (8) at (0, -1) {};
		\node [style=white with black border] (9) at (0, 1) {};
		\node [style=white with black border] (10) at (-1, 0) {};
		\node [style=white with black border] (11) at (2.75, -0.75) {};
	\end{pgfonlayer}
	\begin{pgfonlayer}{edgelayer}
		\draw [style=new edge style 0] (4) to (3);
		\draw [style=new edge style 0] (4) to (5);
		\draw (4) to (11);
		\draw (9) to (1);
		\draw (10) to (1);
		\draw (3) to (1);
		\draw (8) to (1);
	\end{pgfonlayer}
\end{tikzpicture}
\end{center}
\caption{Resolution graph of the $A_{*,o}+A_{*,o}+A_{*,o}+A_{2,**,o}$ singularity.}
\end{figure}
\vspace{5pt}

\subsection{The action of order six} \label{order6_subsection}
Every elliptic curve has an action of order two given by the involution $\iota$, and composing this with the order three action $\tau_3$ gives rise to an action of order six in $E\times E$ by
\begin{align*}
\tau_6\colon\quad E\times E&\longrightarrow E\times E\\
(P,Q)&\longmapsto (P+Q,-P).\\[-5pt]
\end{align*}
Geometrically, this action can be described through the following diagram:
\begin{figure}[H]
\begin{center}
\begin{tikzpicture}
    \begin{axis}[
        axis lines=center,
        xlabel={$x$},
        ylabel={$y$},
        xtick=\empty,
        ytick=\empty,
        xmin=-1.5, xmax=2,
        ymin=-2, ymax=2,
        samples=200,
        domain=-1.5:2,
        restrict y to domain=-2:2,
    ]
    
        \addplot [
            darkred,
            thick,
            samples=300,
            domain=-1:1/2,
        ]
        ({\x}, {sqrt((\x - 1/2)*(\x - 1)*(\x + 1))});

        \addplot [
            darkred,
            thick,
            samples=300,
            domain=-1:1/2,
        ]
        ({\x}, {-sqrt((\x - 1/2)*(\x - 1)*(\x + 1))});
        
          \addplot [
            darkred,
            thick,
            samples=200,
            domain=1:2,
        ]
        ({\x}, {sqrt((\x - 1/2)*(\x - 1)*(\x + 1))});

        \addplot [
            darkred,
            thick,
            samples=200,
            domain=1:2,
        ]
        ({\x}, {-sqrt((\x - 1/2)*(\x - 1)*(\x + 1))});

        \addplot[domain=-2:2] {0.375999 + 0.241225*(0.95 + x)};
        \addplot[domain=-2:2] {-(0.375999 + 0.241225*(0.95 + x))};

        \addplot[dotted] coordinates {(-0.95, 2) (-0.95,-2)};
        \addplot[dotted] coordinates {(0.1, 2) (0.1,-2)};
        \addplot[dotted] coordinates {(1.40819, 2) (1.40819,-2)};

        \node[above left] at (axis cs:-1, 0.375999) {$P$};
        \node[fatpoint] (P) at (axis cs:-0.95, 0.375999) {};
        \node[pointlabel] at (P) {1};
        \node[above right] at (axis cs:0.15, 0.629285) {$Q$};
        \node[fatpoint] (Q) at (axis cs:0.1, 0.629285) {};
        \node[pointlabel] at (Q) {2};
        
        \addplot[only marks, mark=*, mark size=1.5pt] coordinates {(1.40819, 0.944854) (-0.95, -0.375999) (0.1, -0.629285)(1.40819, -0.944854)};
        \node[above left] at (axis cs:1.40819, 0.944854) {$-P-Q$};
        \node[below left] at (axis cs:-0.95, -0.375999) {$-P$};
        \node[below right] at (axis cs:0.1, -0.64) {$-Q$};
        \node[below left] at (axis cs:1.40819, -0.944854) {$P+Q$};
    \end{axis}
\end{tikzpicture} \hspace{25pt}
\begin{tikzpicture}
    \begin{axis}[
        axis lines=center,
        xlabel={$x$},
        ylabel={$y$},
        xtick=\empty,
        ytick=\empty,
        xmin=-1.5, xmax=2,
        ymin=-2, ymax=2,
        samples=200,
        domain=-1.5:2,
        restrict y to domain=-2:2,
    ]
    
        \addplot [
            darkred,
            thick,
            samples=300,
            domain=-1:1/2,
        ]
        ({\x}, {sqrt((\x - 1/2)*(\x - 1)*(\x + 1))});

        \addplot [
            darkred,
            thick,
            samples=300,
            domain=-1:1/2,
        ]
        ({\x}, {-sqrt((\x - 1/2)*(\x - 1)*(\x + 1))});
        
          \addplot [
            darkred,
            thick,
            samples=200,
            domain=1:2,
        ]
        ({\x}, {sqrt((\x - 1/2)*(\x - 1)*(\x + 1))});

        \addplot [
            darkred,
            thick,
            samples=200,
            domain=1:2,
        ]
        ({\x}, {-sqrt((\x - 1/2)*(\x - 1)*(\x + 1))});

        \addplot[domain=-2:2] {0.375999 + 0.241225*(0.95 + x)};
        \addplot[domain=-2:2] {-(0.375999 + 0.241225*(0.95 + x))};

        \addplot[dotted] coordinates {(-0.95, 2) (-0.95,-2)};
        \addplot[dotted] coordinates {(0.1, 2) (0.1,-2)};
        \addplot[dotted] coordinates {(1.40819, 2) (1.40819,-2)};

        \node[below left] at (axis cs:-1, -0.375999) {$-P$};
        \node[fatpoint] (P) at (axis cs:-0.95, -0.375999) {};
        \node[pointlabel] at (P) {2};
        \node[below left] at (axis cs:1.3, -0.944854) {$P+Q$};
        \node[fatpoint] (Q) at (axis cs:1.40819, -0.944854) {};
        \node[pointlabel] at (Q) {1};
        \addplot[only marks, mark=*, mark size=1.5pt] coordinates {(1.40819, 0.944854) (-0.95, 0.375999) (0.1, -0.629285)(1.40819, 0.944854) (0.1, 0.629285)};
        \node[above left] at (axis cs:1.40819, 0.944854) {$-P-Q$};
        \node[above left] at (axis cs:-0.95, 0.375999) {$P$};
        \node[above right] at (axis cs:0.1, 0.64) {$Q$};
        \node[below right] at (axis cs:0.1, -0.64) {$-Q$};
    \end{axis}
\end{tikzpicture}
\caption{Action of $\tau_6$ on $E\times E$.}
\end{center}
\end{figure}
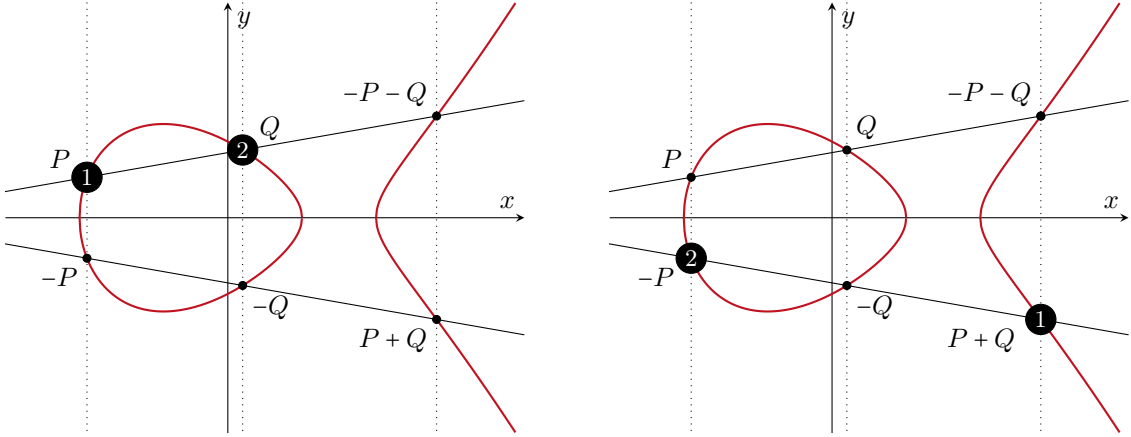

Let $\Ksix$ denote the quotient of $E\times E$ by this action. In a similar spirit as in the order three case, let $P$ and $Q$ be two different points of $E$. Then, the lines $\ell_{P,Q}$ and $\ell_{-P,-Q}$ can both be understood as points in $\mathbb{P}^{2,\vee}$. Now, consider the quotient of $\mathbb{P}^{2,\vee}$ by the action $\tau_2$ that identifies $\ell_{P,Q}$ and $\ell_{-P,-Q}$. One can check that the invariant functions of the coordinate ring of $\mathbb{P}^2$ are generated by two invariants of degree one and one invariant of degree two. Therefore, $\mathbb{P}^2/\tau_2\cong\mathbb{P}(1,1,2)$ and we can construct a map
\begin{align*}
\psi\colon\quad E\times E&\longrightarrow \mathbb{P}(1,1,2)\\
(P,Q)&\longmapsto [\ell_{P,Q}]\\[-5pt]
\end{align*}
Although the construction of this map may seem somewhat indirect, its purpose is clear. As in previous cases, it descends to the quotient, inducing a morphism 
\begin{align*}
\psi_K\colon\quad\Ksix \rightarrow \mathbb{P}(1,1,2).\\[-5pt]
\end{align*}
Generically, the fibre over a point in $\mathbb{P}(1,1,2)$ consists of twelve pairs of points in $E \times E$ (six on each of the lines $\ell_{P,Q}$ and $\ell_{-P,-Q}$), and since six of these are identified in the quotient, we find that $\psi_K$ is a double cover of $\mathbb{P}(1,1,2)$.\\

Assuming that the characteristic of the base field is not two, the ramification locus of $\psi_K$ consists of two components:
\begin{enumerate}
    \item The image of the dual curve $E^\vee$ by the quotient by $\iota$, which corresponds to the vanishing of a weighted degree six polynomial.
    \item The locus where $P = -Q$, i.e., where the line $\ell_{P,Q}$ is vertical. This corresponds to the vanishing of a weighted degree two polynomial in $\mathbb{P}(1,1,2)$.
\end{enumerate}
Hence, the ramification divisor is cut out by a degree eight polynomial, which factors into components of degrees two and six.\\

As before, in characteristic two we can obtain a model by specialising from characteristic zero. This yields a surface of the form
\[
t^2 + g_4(u,v,w) t + f_8(u,v,w) = 0,
\]
where $g_4$ is a degree four polynomial (arising from the ramification divisor), and $f_8$ has degree eight.\\

If the characteristic of the base field is not two or three, then $\Ksix$ has one $A_5$, four $A_2$ and five $A_1$ singularities. As the origin is fixed by the whole group $C_6$, its image becomes the $A_5$ singularity of $\Ksix$. There are fifteen non-trivial $2$-torsion points in $E\times E$, which under the quotient contract $3$-to-$1$ to give rise to the five $A_1$ in $\Ksix$. There are two types of orbits of this kind:
\begin{itemize}
\item On the one hand, we have the points of the form $(P,Q)$ where $P$ and $Q$ are different $2$-torsion points. In the quotient there are two points of this form, and in our model of $\Ksix$ they correspond to the two $A_1$ ambient singularities that appear as a result of $\Ksix$ being embedded in $\mathbb{P}(1,1,2,4)$. These can be found as the intersection of $\Ksix$ with the variety described by the vanishing of the two variables of degree one.

\item On the other hand, we have the three points of the form $(P,P)$ where $P\in E[2]$. As the tangent line at these is vertical, the image of these points in $\mathbb{P}(1,1,2,4)$ is in the intersection of the factor of degree six and the one of degree two.
\end{itemize}

The eight non-trivial $3$-torsion points fixed by $\tau_6^2$ contract $2$-to-$1$ to give the $A_2$ singularities.\\

If the characteristic of the base field is two and $E$ is ordinary, $\Ksix$ has one $E_6^1$, one $D_4^1$ and four $A_2$ singularities. The $E_6^1$ corresponds to the $C_6$-action at the origin. Using Magma, we can see that the intersection matrix of the exceptional curves of the desingularisation of these singularities form an $E_6$ configuration, and as the Tjurina number is six, we deduce that the singularity is of type $E_6^1$. The $D_4^1$ come from the fact that the three non-trivial $2$-torsion points contract $3$-to-$1$ from $\Ktwo$ to $\Ksix$ and the four $A_2$ because the eight $3$-torsion points contract $2$-to-$1$ from $\Kthree$ to $\Ksix$. \\

If the characteristic of the base field is two and $E$ is supersingular, we know that $(E\times E)/C_6$ is not a K3 surface, as it is a quotient of $(E\times E)/\langle \iota\rangle$, which is rational.\\

If the characteristic of the base field is three and $E$ is ordinary, $\Ksix$ has one $E_7^1$, one $E_6^1$ and five $A_1$ singularities. We check with Magma that the singularity at the origin is $E_7^1$ by resolving it and checking that the Tjurina number is seven, which rules out the possibility of it being of type $E_7^0$.  The $E_6^1$ singularity is the result of the fact that the two $E_6^1$ singularities of $\Kthree$ contract to one. The description of the $A_1$ singularities is identical to the characteristic zero case.\\

If the characteristic of the base field is three and $E$ is supersingular, $(E\times E)/C_6$ still has five $A_1$, but instead of one $E_6^1$ and one $E_7^1$ rational double point, it has an elliptic singularity of type $Tr$ in Laufer's notation \cite{Laufer1977OnSingularities}.\newpage Hence, by Proposition \ref{ellsing_prop}, $(E\times E)/C_6$ is a rational surface.
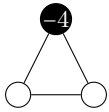
\begin{figure}[H]
\begin{center}
\begin{tikzpicture}
	\begin{pgfonlayer}{nodelayer}
		\node [style=black with white font] (0) at (2.5, 2) {$-4$};
		\node [style=white with black border] (1) at (3, 1) {};
		\node [style=white with black border] (2) at (2, 1) {};
	\end{pgfonlayer}
	\begin{pgfonlayer}{edgelayer}
		\draw [style=new edge style 0] (0) to (2);
		\draw [style=new edge style 0] (2) to (1);
		\draw (1) to (0);
	\end{pgfonlayer}
\end{tikzpicture}
\caption{Resolution graph of the $Tr$ singularity.}
\end{center}
\end{figure}

One thing that may not be apparent from the dual graph is that the exceptional lines all meet in one point, but this information can be extracted using Magma.

\subsection{Elliptic curves with complex multiplication} \label{examples_ell_subsection}

A major difference between the cyclic cases and the quotients we are about to describe is that, whereas in the previous cases we could construct positive-dimensional families of examples, there are only finitely many quotients by the actions of $Q_8$, $Q_{12}$, and $\SL_2(\FF_3)$. Many of these group actions were first described by Fujiki, who studied them in the context of automorphisms of complex tori \cite{Fujiki1988FiniteTwo}.\\

Another important point is that none of these groups is a subgroup of $M_2(\Z)$, since it is known that the only finite subgroups of $\GL_2(\Z)$ are $C_2$, $C_3$, $C_2^2$, $C_4$, $C_6$, $S_3$, $D_4$, and $D_6$ \cite{Voskresenskii1967OnII}. Therefore, to realise the actions by these groups, $A$ has to be the product of curves with complex multiplication.\\

The simplest examples of elliptic curves with complex multiplication are those with $j$-invariants $0$ and $1728$, which have multiplication by $\zeta_3$ and $i$, respectively. In these two cases, Fujiki showed that $E_{\zeta_3} \times E_{\zeta_3}$ and $E_i \times E_i$ both admit rigid and symplectic actions by all three groups $Q_8$, $Q_{12}$, and $\SL_2(\FF_3)$ \cite[Theorem 4.2]{Fujiki1988FiniteTwo}.\\

However, neither of these curves has good ordinary reduction at the primes $p=2$ or $3$, so they cannot be used to construct examples of the generalised Kummer surfaces $A/Q_{12}$ in characteristic two or $A/\SL_2(\FF_3)$ in characteristic three. For this reason, we now present two less common examples of complex multiplication elliptic curves, found in the LMFDB \cite{TheLMFDBCollaboration2025TheDatabase}.\\

\subsubsection{The \texorpdfstring{curve with complex multiplication by $\sqrt{2}i$}{first curve with complex multiplication}}
Let $E_{\sqrt{2}i}$ denote the unique elliptic curve (up to isomorphism) with complex multiplication by $\sqrt{2}i$. It has $j$-invariant $8000$, and a model over $\Z$ is
\begin{align*}
y^2 = x^3 - x^2 - 3x - 1.\\[-10pt]
\end{align*}
The map corresponding to multiplication by $\sqrt{2}i$ is given by the $2$-isogeny:
\begin{align*}
[\sqrt{2}i]:\quad E_{\sqrt{2}i}&\longrightarrow E_{\sqrt{2}i}\\
(x,y)&\longmapsto\left(-\frac{x^2+1}{2(x+1)},\frac{\sqrt{2}i(x^2+2x-1)y}{4(x+1)^2}\right)\\[-10pt]
\end{align*}
The curve $E_{\sqrt{2}i}$ has bad reduction only at $p=2$, and has good ordinary reduction at $p=3$. The reduction at three defines a curve over $\FF_3$ whose Frobenius endomorphism has characteristic polynomial $x^2 - 2x + 3$, and thus acts as multiplication by $1 + \sqrt{2}i$.\\

\subsubsection{The \texorpdfstring{curve with complex multiplication by $\tfrac{1}{2}(1 + \sqrt{7}i)$}{second curve with complex multiplication}} 
Let $E_\rho$ denote the unique elliptic curve (up to isomorphism) with complex multiplication by $\rho = \tfrac{1}{2}(1 + \sqrt{7}i)$. It has $j$-invariant $-3375$, and a model over $\Z$ is
\begin{align*}
y^2 + xy = x^3 - x^2 - 2x - 1.\\[-5pt]
\end{align*}
\newpage
The map corresponding to multiplication by $\rho$ is given by the $2$-isogeny:\\

\begin{adjustbox}{width=\textwidth}
\begin{minipage}{1.1\textwidth}
\begin{align*}
[\rho]:\quad E_\rho&\longrightarrow E_\rho\\
(x,y)&\longmapsto\left(\frac{(\rho-2)x^2-2\rho}{4x+\rho+2},-\frac{((2\rho+4)x^2+(5\rho+2)x-4\rho+8)y+(3\rho-2)(x^3+x^2-2x-1)}{16 x^2+(8\rho+16)x+2}\right)\\[5pt]
\end{align*}
\end{minipage}
\end{adjustbox}

The curve  $E_\rho$ has bad reduction only at $p=7$, and good ordinary reduction at $p=2$. Its reduction at two defines a curve over $\FF_2$ with characteristic polynomial $x^2 - x + 2$. Therefore, the Frobenius acts on this curve as multiplication by $\tfrac{1}{2}(1+\sqrt{7}i)$.\\

\subsection{The actions by \texorpdfstring{$Q_8$}{Q8}}
Let $A=E_{\sqrt{2}i}\times E_{\sqrt{2}i}$. We will now construct two different actions of $Q_8$ on $A$ such that the quotients by the two actions have different singular points.\\

$Q_8$ can be described by the presentation $Q_8=\langle \sigma_4,\tau_4\,|\,\sigma_4^4,\sigma_4^2\tau^2_4,\sigma_4\tau_4\sigma_4\tau_4^{-1}\rangle$  
and, thus, we can see that $Q_8$ acts on $A$ by considering the following two automorphisms:
\begin{align*}
\sigma_4\colon\quad E_{\sqrt{2}i}\times E_{\sqrt{2}i}&\longrightarrow E_{\sqrt{2}i}\times E_{\sqrt{2}i} &\tau_4\colon\quad E_{\sqrt{2}i}\times E_{\sqrt{2}i}&\longrightarrow E_{\sqrt{2}i}\times E_{\sqrt{2}i}  \\
(P,Q)&\longmapsto ([\sqrt{2} i]P-Q,-P-[\sqrt{2} i]Q) &(P,Q)&\longmapsto (-Q,P)&\\[-5pt]
\end{align*}

The quotient by this action $(E_{\sqrt{2}i}\times E_{\sqrt{2}i})/Q_8$ has two $D_4$, three $A_3$ and two $A_1$ singularities.\\

We saw in Subsection \ref{order4_subsection} that the fixed points of $\tau_4$ are the subgroup of points of the form $(P,P)$ with $P\in E_{\sqrt{2}i}[2]$, i.e. $P\in \{O,(-1,0),(1-\sqrt{2},0),(1+\sqrt{2},0)\}$. It takes a bit more work to check that the fixed points of $\sigma_4$ are 
$$\{O,((-1,0),(-1,0)),((1-\sqrt{2},0),(1+\sqrt{2},0)),((1+\sqrt{2},0),(1-\sqrt{2},0))\}.$$

Therefore, two points are fixed by $Q_8$: $O$ and $((-1,0),(-1,0))$. The images of these two points by the quotient map are the two $D_4$ singularities.\\

One can check that there are three subgroups of $Q_8$ isomorphic to $C_4$: $\langle\sigma_4\rangle$, $\langle\tau_4\rangle$ and $\langle \sigma_4\tau_4 \rangle$. One can check that each of these groups fix $O$ and $((-1,0),(-1,0))$ and two other different $2$-torsion points. So, in total, there are six $2$-torsion points that are fixed by a unique copy of $C_4$, and through the quotient map, they map $2$-to-$1$ to the three $A_3$ singularities. The eight remaining $2$-torsion points of $E_{\sqrt{2}i}\times E_{\sqrt{2}i}$ that are not fixed by any copy of $C_4$ map $4$-to-$1$ to the $A_1$ singularities.\\

A different action of $Q_8$ on $A$ is the one generated by the following two automorphisms:
\begin{align*}
\varsigma_4\colon\quad E_{\sqrt{2}i}\times E_{\sqrt{2}i}&\longrightarrow E_{\sqrt{2}i}\times E_{\sqrt{2}i} \\
(P,Q)&\longmapsto (P+[{\sqrt{2}i}]Q,[{\sqrt{2}i}]P-Q)\\[10pt]
\upsilon_4\colon\quad E_{\sqrt{2}i}\times E_{\sqrt{2}i}&\longrightarrow E_{\sqrt{2}i}\times E_{\sqrt{2}i}  \\
(P,Q)&\longmapsto ((-1+[{\sqrt{2}i}])P-2Q,-[{\sqrt{2}i}]P+(1-[{\sqrt{2}i}])Q)\\[-5pt]
\end{align*}

The quotient by this action $(E_{\sqrt{2}i}\times E_{\sqrt{2}i})/Q_8$ has four $D_4$ and three $A_1$ singularities.\\

The fixed points of $\varsigma_4$ are those of the form $(P,Q)\in(E_{\sqrt{2}i}\times E_{\sqrt{2}i})[2]$, where both $P$ and $Q$ are in the kernel of $[\sqrt{2} i]$, so $P,Q\in \{O,(-1,0)\}$. It can be shown that these same points fixed by $\upsilon_4$, and therefore, by the whole $Q_8$. The image of these points by the quotient map would be the four $D_4$ singularities.\\

The other twelve $2$-torsion points are not fixed by any copy of $C_4<Q_8$ and, therefore, they map $4$-to-$1$ to the three $A_1$ singularities in the quotient.\\

\subsection{The action by \texorpdfstring{$Q_{12}$}{Q12}}
Let $A=E_\rho\times E_\rho$, and consider the automorphisms
\begin{align*}
\tau_4\colon\quad E_\rho\times E_\rho&\longrightarrow E_\rho\times E_\rho & \tau_6\colon\quad E_\rho\times E_\rho&\longrightarrow E_\rho\times E_\rho \\
(P,Q)&\longmapsto (P+[\rho]Q,([\rho]-1)P-Q) &(P,Q)&\longmapsto (P+Q,-P)\\[-5pt]
\end{align*}
They satisfy that $\tau_6^3\tau_4^2=\tau_4^4=\tau_6\tau_4\tau_6\tau_4^{-1}=\mathrm{id}$ and therefore, they generate $Q_{12}$.\\

In characteristics not two or three, 
$(E_\rho\times E_\rho)/Q_{12}$ has one $D_5$, three $A_3$, two $A_2$ and one $A_1$ singularity. \\

 It is easy to see that the image under the quotient map of $O$ is the $D_5$ singularity, as locally this map is the quotient of $\mathbb{A}^2$ by the action of $Q_{12}$. One can check that the points fixed by $\tau_4$ are those of the form $(P,Q)$ where $Q\in E_\rho$ is in the kernel of $[\rho]$, and $P\in E_\rho$ is in the kernel of $([\rho]-1)$. Note that $([\rho]-1)$ corresponds to multiplication by $\tfrac{1}{2}(-1+\sqrt{7}i)$, which is also a $2$-isogeny, so
\begin{align*}
\langle P\rangle&=\{O, \left(\tfrac{1}{8}(-5+\sqrt{7}i),\tfrac{1}{16}(5-\sqrt{7}i) \right)\}, & \langle Q\rangle=\{O,\left(\tfrac{1}{8}(-5-\sqrt{7}i),\tfrac{1}{16}(5+\sqrt{7}i)\right) \},
\end{align*}
 are both copies of $\Z/2\Z$ inside $E_\rho[2]$.\\
 
 Therefore, the points fixed by $\tau_4$ are a subgroup $(\Z/2\Z)^2<(E_\rho\times E_\rho)[2]$. Now, there are two other subgroups of order four inside $Q_{12}$: $\langle\tau_4\tau_6^2\rangle$ and $\langle\tau_4\tau_6^4\rangle$, and each of these groups fixes a different $(\Z/2\Z)^2< (E_\rho\times E_\rho)[2]$. Hence, in total, there are nine non-trivial $2$-torsion points fixed by a copy of $C_4$ inside $Q_{12}$. They are identified $3$-to-$1$ in the quotient and their images are the three $A_3$ singularities. There are six other $2$-torsion points, which are permuted by the actions of $\tau_4$ and $\tau_6^2$, so under the quotient map, they map $6$-to-$1$ to the $A_1$ singularity.\\

 As for the two $A_2$ singularities, as we discussed in Subsection \ref{order6_subsection}, there is a subgroup of order nine of $(E_\rho\times E_\rho)[3]$ fixed by $\tau_6^2$, and the action of $\tau_4$ on the eight non-trivial gives rise to two orbits, whose images under the quotient map are the two $A_2$ singularities.\\

If the characteristic is two, the picture is quite similar. In that case, $(E_\rho\times E_\rho)/Q_{12}$ has one $E_8^4$, one $E_7^3$, and two $A_2$ singularities.\\

The image under the quotient map of $O$ can be checked to be an $E_8^4$ singularity. This can be seen from the fact that $C_6$ is a normal subgroup of $Q_{12}$, and therefore there is a quotient map \[\pi:(E_\rho\times E_\rho)/C_6\rightarrow(E_\rho\times E_\rho)/Q_{12}.\] At the image of the origin, $(E_\rho\times E_\rho)/C_6$ has a singularity of type $E_6^1$, and $\pi$ is an unramified double cover, so by the work of Artin, we deduce that the corresponding singularity of $(E_\rho\times E_\rho)/Q_{12}$ is of type $E_8^4$ \cite[Case 4]{Artin1975CoveringsP}.\\

The $2$-torsion of $E_\rho$ is the kernel of Frobenius, and therefore, from the discussion in Subsection \ref{examples_ell_subsection}, the four $2$-torsion points of $(E_\rho\times E_\rho)[2]$ are the fixed points of $\tau_4$. Furthermore, these points are also fixed by $\tau_4\tau_6^2$ and $\tau_4\tau_6^4$, so they are the fixed points of all copies of $C_4<Q_{12}$. The other three non-trivial $2$-torsion points are identified $3$-to-$1$ in the quotient, yielding an $E_7^3$ singularity. The fact that this is the right type of singularity can be deduced from the analysis of the quotients by the group $C_4$ in characteristic two that we made in Subsection \ref{order4_subsection}. Finally, the two $A_2$ singularities come from the image of $3$-torsion points, in the same way as in the characteristic $p\neq 2$ case.\\

\subsection{The action by \texorpdfstring{$\SL_2(\FF_3)$}{SL(2,3)}}
Let $A=E_{\sqrt{2}i}\times E_{\sqrt{2}i}$, and consider the automorphisms
\begin{align*}
\tau_3\colon\quad E_{\sqrt{2}i}\times E_{\sqrt{2}i}&\longrightarrow E_{\sqrt{2}i}\times E_{\sqrt{2}i} &\varsigma_4\colon\quad E_{\sqrt{2}i}\times E_{\sqrt{2}i}&\longrightarrow E_{\sqrt{2}i}\times E_{\sqrt{2}i}  \\
(P,Q)&\longmapsto (-P-Q,P) &(P,Q)&\longmapsto (P+[{\sqrt{2}i}]Q,[{\sqrt{2}i}]P-Q)\\[-5pt]
\end{align*}
Using Magma, one can check that these two automorphisms generate $\SL_2(\FF_3)$.\newpage One can also check that the automorphism $\upsilon_4$ that we described in Subsection \ref{examples_ell_subsection} is $\upsilon_4=\tau_3\varsigma_4\tau_3^2$ and $Q_8\cong\langle\varsigma_4,\upsilon_4 \rangle$ is a normal subgroup of $\SL_2(\FF_3)$.\\

In characteristics not two or three, 
$(E_{\sqrt{2}i}\times E_{\sqrt{2}i})/\SL_2(\FF_3)$ has one $E_6$, one $D_4$, four $A_2$ and one $A_1$ singularity. \\

Recall that the elements fixed by $Q_8$ were those  inside $(E_{\sqrt{2}i}\times E_{\sqrt{2}i})[2]$ of the form $(P,Q)$ where $P, Q\in \{O,(-1,0)\}$. Then, the image under the quotient map of $O$ is the $E_6$ singularity, and the other three points fixed by $Q_8$ map $3$-to-$1$ to the $D_4$ singularity. The remaining twelve $2$-torsion points are mapped $12$-to-$1$ to the $A_1$ singularity.\\

Inside $\SL_2(\FF_3)$, there are four subgroups of order three: $\langle \tau_3\rangle$, $\langle\varsigma_4\tau_3\rangle$, $\langle\tau_3\varsigma_4\rangle$, and $\langle \varsigma_4\tau_3\varsigma_4^3\rangle$.
 For each of these groups, the set of fixed points is a subgroup of order nine of $(E_{\sqrt{2}i}\times E_{\sqrt{2}i})[3]$. The action of $Q_8=\langle\varsigma_4,\upsilon_4\rangle$ on the thirty-two non-trivial elements gives rise to four orbits, whose images under the quotient map are the four $A_2$ singularities.\\
 
In characteristic three, 
$(E_{\sqrt{2}i}\times E_{\sqrt{2}i})/\SL_2(\FF_3)$ has one $E_8^2$, one $E_6^1$, one $D_4$ and one $A_1$ singularity. \\

The image under the quotient map of $O$ can be checked to be an $E_8^2$ singularity, by studying the quotient map $\pi:(E_{\sqrt{2}i}\times E_{\sqrt{2}i})/Q_8\rightarrow(E_{\sqrt{2}i}\times E_{\sqrt{2}i})/\SL_2(\FF_3)$. At the image of the origin, $(E_{\sqrt{2}i}\times E_{\sqrt{2}i})/Q_8$ has a singularity of type $D_4$, and $\pi$ is an unramified triple cover, so by the work of Artin, we deduce that the corresponding singularity of $(E_{\sqrt{2}i}\times E_{\sqrt{2}i})/\SL_2(\FF_3)$ is of type $E_8^2$ \cite[Section 5]{Artin1975CoveringsP}.\\

As before, we have a subgroup of order four of the $2$-torsion that is fixed by $Q_8$. The three non-trivial points again map $3$-to-$1$ to the $D_4$ singularity and the twelve other $2$-torsion points map $12$-to-$1$ to the $A_1$ singularity.\\

Finally, there are eight non-trivial $3$-torsion points that are fixed by the order three subgroups of $\SL_2(\FF_3)$. They get mapped $8$-to-$1$ in the quotient, producing an $E_6^1$ singularity, as discussed in Subsection \ref{order3_subsection}.

\bibliographystyle{alpha}
\bibliography{references.bib}

\end{document}